\documentclass[11pt,a4paper]{amsart}  

\usepackage[a4paper]{geometry}
\usepackage[utf8]{inputenc}
\usepackage[english]{babel}

\usepackage{amsthm}
\usepackage{amstext}
\usepackage{amsmath}
\usepackage{amsfonts}
\usepackage{amssymb}
\usepackage{mathtools}

\usepackage{amsmath, amssymb, amsfonts}
\usepackage[arrow, matrix, curve]{xy}
\usepackage{enumitem} 
\usepackage{color}
\usepackage[hypertexnames=false]{hyperref}
\usepackage[utf8]{inputenc}
\usepackage{todo}
\usepackage{mathdots}
\usepackage{appendix}
\usepackage{pdflscape}

\usepackage{MnSymbol}

\newtheorem{prop}{Proposition}[section]
\newtheorem{conj}{Conjecture}[section]
\newtheorem{lemma}[prop]{Lemma}
\newtheorem{corollary}[prop]{Corollary}

\newtheorem{thm}[prop]{Theorem}
\newtheorem{example}[prop]{Example}

\newtheorem*{thm*}{Theorem}
\newtheorem*{prop*}{Proposition}
\newtheorem*{definition*}{Definition}

\theoremstyle{definition}
\newtheorem{definition}[prop]{Definition}

\theoremstyle{remark}
\newtheorem{rem}[prop]{Remark}

\newcommand{\Z}{\mathbb{Z}}
\newcommand{\N}{\mathbb{N}}
\newcommand{\R}{\mathbb{R}}

\newcommand{\A}{\mathbb{A}}
\newcommand{\Gm}{\mathbb{G}_m}
\newcommand{\Projsp}{\mathbb{P}}

\newcommand{\rg}{\mathcal{O}}
\newcommand{\res}{\kappa}
\newcommand{\cres}{\overline{\kappa}}
\newcommand{\unif}{\pi}
\newcommand{\quot}{\mathrm{K}}

\newcommand{\im}[1]{\mathrm{Im}(#1)}

\newcommand{\gln}{\mathrm{GL}_n}

\newcommand{\affw}{W^{\mathrm{aff}}}

\newcommand{\grs}{\mathrm{Gr}^{n,r}}

\newcommand{\gr}[1]{\mathrm{Gr}^k\left({#1}\right)}

\newcommand{\can}{\mathcal{M}}
\newcommand{\cangoe}{\widetilde{\mathcal{G}}}
\newcommand{\cangen}{\mathcal{G}}
\newcommand{\cangenpl}{\cangen^{\mathrm{pl}}}
\newcommand{\cangenstr}{\mathcal{S}}
\newcommand{\cangenstrpl}{\mathcal{S}^{\mathrm{pl}}}

\newcommand{\sset}{\binom{[n]}{k}}

\newcommand{\defeq}{\mathrel{\mathop:}=}
\newcommand{\loc}{\mathcal{M}^{\text{loc}}}

\newcommand{\spec}[1]{\mathrm{Spec}\left(#1\right)}
\newcommand{\bl}[2]{{\mathrm{Bl}_{#1}\left(#2\right)}}

\newcommand{\Bl}[2]{\mathrm{BL}_{#1}\left(#2\right)}

\newcommand{\tot}[1]{#1^{\mathrm{tot}}}
\newcommand{\str}[1]{#1^{\mathrm{s}}}

\newcommand{\sing}[1]{#1^{\mathrm{sing}}}
\newcommand{\schuvar}[1]{X_{#1}}
\newcommand{\schucell}[1]{\schuvar{#1}^\circ}

\newcommand{\M}[1]{\mathcal{M}_{\Projsp}\left(#1\right)}
\newcommand{\Mgr}[2]{\mathcal{M}_{\mathrm{Gr}^#1}\left(#2\right)}

\newcommand{\projectbar}[1]{{\overline{\pr_{\Projsp}^{#1}}}}
\newcommand{\projbar}{{\overline{\pr_{\Projsp}}}}
\newcommand{\project}[1]{{\pr_{\Projsp}^{#1}}}
\newcommand{\pr}{\mathrm{pr}}

\newcommand{\projectgr}[1]{{\pr_{\mathrm{Gr}}^{#1}}}
\newcommand{\muspro}[2]{\pr_{#1 , #2}}
\newcommand{\musprol}[1]{\pr_{#1}}

\title{local Models, Mustafin varieties and semi-stable resolutions}

\begin{document}
\author{Felix Gora}
\title{Local models, Mustafin varieties and semi-stable resolutions}

\maketitle

\begin{abstract}
	Our goal is to analyse singularities of integral models of Shimura varieties. One approach is to construct local models, which model the singularities of the corresponding integral model using linear algebra dada and find resolutions with mild singularities thereof. More precisely we will attack the question of existence of semi-stable resolutions. We will discuss an approach developed by Genestier. In this approach a candidate for a semi-stable resolution was given as the blow-up of a Grassmannian variety in Schubert varieties of its special fiber. Explicit calculations show that this approach does not work in general. Using the flatness of the local models, we describe these local models as Mustafin varieties for Grassmannian varieties. We combine several results on the structure of Mustafin varieties for projective spaces with the Pl\"ucker embedding to construct a candidate for a semi-stable resolution of local models. Under some additional assumptions this candidate generalises the approach suggested by Genestier. Furthermore under the same assumptions the new candidate agrees with the semi-stable resolution constructed by G\"ortz for small dimensions.
\end{abstract}

\section*{Introduction}
 In the study of Shimura varieties it is of great interest to construct models over the ring of integers $\rg$ of the completion of the reflex field at a prime  with finite residue characteristic $p$. These models should at least be flat and ideally have mild singularities. The special case of Shimura varieties of PEL type are moduli spaces of abelian varieties with some extra structures (polarisation, endomorphism and level structure). For parahoric level structures candidates for such models are constructed by Rapoport and Zink in \cite{RZ96} by posing the moduli problem over $\rg$.\\
 In the attempt to analyse the occurring singularities they define so called \emph{local models}. These models are constructed as projective varieties over $\rg$ and model the singularities of the integral models. More precisely every point in the integral model has an \'{e}tale neighbourhood isomorphic to an \'{e}tale neighbourhood of the corresponding local model. The advantage of local models is that they are cut out in a product of Grassmannian varieties by equations arising from linear algebra and hence are easier to handle. Although these models are not flat in general as pointed out by Pappas in \cite{Pa00}, it was proven by G\"ortz that local models in the so called \emph{linear case} and the \emph{symplectic case}  are flat (cf. \cite{G01} and \cite{Goe03}). Further developments of local models for other cases and the question of their flatness can be found in \cite{PRS13}.\\
  In the following we will focus on local models in the linear case with Iwahori level structure and study their singularities in this case. Below we will give a precise definition of the local model corresponding to this data. 

\vspace{13pt}

Now let $\rg$ be any complete discrete valuation ring with uniformizer $\unif$, quotient field $\quot$ and resedueclass field $\res$.
For this case the local model over $\rg$ is constructed as follows.
Fix two natural numbers $k<n$. To shorten the notation we denote the set $\{0,\dots,n-1\}$ by $[n]$. Furthermore we fix the canonical basis $\{e_i\}_{i\in [n]}$ of $\quot^n$. For $i\leq n-1$ we denote by $\Lambda_i$ the lattice generated by the elements $\unif^{-1}e_0, \dots , \unif^{-1}e_{i-1},e_{i},\dots ,e_{n-1}$ and define the standard lattice chain $\Gamma^{\mathrm{st}}$ to be
	\begin{align*}
		 \dots\rightarrow \Lambda_0 \rightarrow \Lambda_1\rightarrow \dots\rightarrow \Lambda_n=\unif^{-1} \Lambda_0\rightarrow\dots
	\end{align*} 
	For an $\rg$-scheme $S$ we write $\Lambda_{i,S}$ for $\Lambda_i\otimes_{\rg}\mathcal{O}_S$ and define the $S$-valued points of the functor $\loc$ to be diagrams of the form
\[\begin{xy}
	\xymatrix{
		 \Lambda_{0,S}\ar[r]&\Lambda_{1,S}\ar[r]& \dots \ar[r] & \Lambda_{n-1,S}\ar^{\unif}[r]& \Lambda_{0,S}\\
		 \mathcal{F}_0\ar@{^(->}[u]\ar[r] & \mathcal{F}_1\ar@{^(->}[u]\ar[r]&\dots \ar[r] & \mathcal{F}_{n-1}\ar@{^(->}[u]\ar[r] & \mathcal{F}_0 \ar@{^(->}[u]
	}
	\end{xy}\]
where the $\mathcal{F}_i$'s are locally free $\mathcal{O}_S$-submodules of $\Lambda_{i,S}$ of rank $k$ that are Zariski-locally direct summands.\\
This functor is represented by a closed subscheme of the product $\prod_{i\in[n]}\gr{\Lambda_i}$ of Grassmanian varieties.
 We can easily identify the generic fiber $\loc_{\quot}$ with the Grassmannian $\grs_{\quot}$, but the special fiber is much more complicated.\\
 
 \vspace{2pt}
  Let us illustrate some of the behaviour in the case $n=2$ and $k=1$ (cf. \cite[Section 4.4]{Ha05}). If we further assume $\rg=\Z_p$ then $\loc$ models the singularities of the modular curve endowed with $\Gamma_0(p)$-level structure.\\ Fix a $\rg$-algebra $R$. To simplify the notation let us identify $\Lambda_{0,R}$ and $\Lambda_{1,R}$ with $R\oplus R$. The $R$-valued points $\loc(R)$ are now given by commutative diagrams of the form
\[\begin{xy}
	\xymatrix{
		 R\oplus R \ar[rr]^{\begin{bmatrix} \unif & 0 \\ 0 & 1 \end{bmatrix} }					&&R\oplus R\ar[rr]^{\begin{bmatrix} 1 & 0 \\ 0 & \unif \end{bmatrix} }						&&  R\oplus R \\
		 \mathcal{F}_0\ar@{^(->}[u]\ar[rr] 	&& \mathcal{F}_1\ar@{^(->}[u]\ar[rr] 	&& \mathcal{F}_{0}\ar@{^(->}[u] 
	}
	\end{xy}\]
where $\mathcal{F}_i$ is an element in $\Projsp^1(R)$ for $i=1,2$.
Let us fix local coordinates and take a pair $(\mathcal{F}_0,\mathcal{F}_1)\in \Projsp^1(R)\times \Projsp^1(R)$. First we note that if $\mathcal{F}_0$ is represented by a homogeneous column vector $[x:1]^{\mathrm{t}}$ then the image $[\unif x:1]^{\mathrm{t}}$ again represents an element in $\Projsp^1(R)$ hence has to coincide with $\mathcal{F}_1$. In particular this chart of the local model can be identified with $\A^1_{\rg}$. Now let us assume that $\mathcal{F}_0$ is represented by a homogeneous vector of the form $[1:x]^{\mathrm{t}}$ and $\mathcal{F}_1$ is represented by $[y:1]^{\mathrm{t}}$. In particular we see that the pair $(\mathcal{F}_0,\mathcal{F}_1)$ is in $\loc(R)$ precisely when $xy=\unif$. Hence this chart of the local model is isomorphic to $\spec{\rg [x,y]/(xy-\unif) }$. Gluing the charts lets us identify $\loc$ with the blow-up of $\Projsp^1_{\rg}$ in the origin of the special fiber. In particular the special fiber $\loc_{\res}$ is consists of two projective lines $\Projsp^1_{\res}$ intersecting transversally in one point.\\

\vspace{2pt}

Generalising the type of singularities of the example above leads to the notion of semi-stability defined below. More detailed discussions of this definition can be found for example in \cite{dJ96} or \cite{Har01}.
\begin{definition*}
	For a complete discrete valuation ring $\rg$ with uniformizer $\unif$ we call an $\rg$-variety $X$ semi-stable if \'{e}tale locally $X$ is of the form 
	\begin{align*}
		\spec{\rg[x_0,\dots,x_r]/(\prod_{i\leq m} x_i-\unif  )}
	\end{align*}  
	for some $r$ and $m$.
\end{definition*} 

As a generalisation of the example above it is well known that the local models $\loc$ in the so called \emph{Drinfeld case} (cf. \cite{RZ96}), i.e for $k=1$ and $n$ arbitrary, are semi-stable (see \cite[Section 3.69]{RZ96} cf. also  \cite{FA01} or \cite{Mu78}). In \cite{FA01} Faltings also constructs toroidal resolutions for $k=2$.\\
 We can also define a symplectic version of the local model cf. \cite{GEN00}. This version is obtained by imposing a certain self-duality condition in the moduli description above. In loc. cit. a semi-stable resolution of the local model in the symplectic case for $n\leq 6$ was constructed. It was also suggested in loc. cit. remark 3 at the end of section 2 that a similar construction produces a semi-stable resolution $\cangen\rightarrow\loc$ in the linear case. Let us explain this construction for the linear rather than the symplectic case. \\
	Set $\cangen_0\defeq \gr{\Lambda_0}$ and inductively for $1\leq i<(n-k)k$ define $\cangen_i$ to be the blow-up of $\cangen_{i-1}$ in the union of the strict transforms of the Schubert varieties of dimension $i-1$ in the special fiber of $\gr{\Lambda_0}$. The last blow-up $\cangen_{(n-k)k-1}$ will be denoted by $\cangen$.\\
  A semi-stable resolution $\cangoe\rightarrow\loc$ for $n\leq 5$ at least for an open neighbourhood of the "most singular point" was given in \cite{G04}. Other cases of local models and their resolutions were studied in \cite{R13}, \cite{Kr03} and \cite{PR05}.

\vspace{13pt}
Starting with the observation that the candidate $\cangen$ in \cite{GEN00} does not factor through $\loc$ for $n=5$ and $k=2$, our goal is to construct a candidate for a semi-stable resolution for arbitrary $n$ and $k$. Therefore we consider the strict transform $\cangenstr$ of the projection $\loc\rightarrow \gr{\Lambda_0}$ under the blow-up $\cangen\rightarrow\gr{\Lambda_0}$. For this strict transform we can show the following theorem:
 \begin{thm*}
 	For $n=5$ and $k=2$ the blow-up $\cangenstr$ is a semi-stable resolution of $\loc$. By passing to a neighbourhood of the worst singularity of $\loc$ one recovers the local semi-stable resolution defined in \cite{G04}.
 \end{thm*}
If $\cangen\rightarrow\gr{\Lambda_0}$ factors through $\loc$, the projection $\cangenstr\rightarrow \cangen$ is an isomorphism. We thus have shown, that $\cangenstr$ is a better candidate for a semi-stable resolution. \\ 
But since it is hard to show the semi-stability of $\cangenstr$ (cf. \cite{G04}), we will adapt the idea of blowing up $\gr{\Lambda_0}$ and construct a candidate $\can$ for a semi-stable resolution as a blow-up $\can\rightarrow \gr{\Lambda_0}$ but with slightly different centers. The advantage of this approach is, that in some cases we might be able to use Lemma \ref{lem_gen_semi-stable} by \cite{GEN00} showing that semi-stability is preserved under blow-ups provided that the centers of the blow-ups are sufficiently nice. \\
 In contrast to loc. cit. we will use the language of Mustafin varieties and the recent results on their behaviour (cf. \cite{CHSW11}, \cite{AL17}). Since $\loc$ is flat (see \cite{G01}), we can identify it with the closure of the generic fiber embedded into $\prod_{\Lambda_i\in\Gamma^{\mathrm{st}}} \gr{\Lambda_i}$, which is by definition the Mustafin variety $\Mgr{k}{\Gamma^{\mathrm{st}}}$. The main idea for the construction of $\can$ is to use the Pl\"ucker embedding $\gr{\Lambda_0}\rightarrow\Projsp(\bigwedge^k \Lambda_0 )$ and a compatible embedding of $\loc$ into the Mustafin variety $\M{\bigwedge^k\Gamma^{\mathrm{st}}}$ where $\bigwedge^k\Gamma^{\mathrm{st}}$ denotes the set $\{\bigwedge^k\Lambda_i\vert \Lambda_i\in\Gamma^{\mathrm{st}} \}$. We expect that these two Mustafin varieties have the same number of irreducible components of their special fibers. In this case we describe an explicit bijection of the sets of irreducible components in Conjecture \ref{conj_convexcl_bijection} (see Section \ref{sec_conj} for a more detailed discussion). We prove Conjecture \ref{conj_convexcl_bijection} for $k\leq 2$ and $n$ arbitrary and we also check the cases $n\leq 7$ and $k$ arbitrary by computer. It would be interesting to study whether the restiction $k\leq 2$ in \cite{FA01} is related to a similar statement. Assuming the conjecture, we can show the following behaviour of the embedding $\Mgr{k}{\Gamma^{\mathrm{st}}}\rightarrow \M{\bigwedge^k\Gamma^{\mathrm{st}}}$.
 \begin{prop*}
 	Assume Conjecture \ref{conj_convexcl_bijection}. Then we have a bijection
	\begin{align*}
		\left\{C\middle\vert C \text{ irr. component of } \M{\bigwedge^k\Gamma^{\mathrm{st}}}_{\res} \right\} &\longrightarrow \left\{C\middle\vert C \text{ irr. component of } \Mgr{k}{\Gamma^{\mathrm{st}}}_{\res} \right\}\\
		C&\longmapsto C\cap \Mgr{k}{\Gamma^{\mathrm{st}}}_{\res}
	\end{align*}	
	between the sets of irreducible components of the special fibres.
 \end{prop*}
 We will prove that a Mustafin variety $\M{\Gamma}$ is semi-stable if $\Gamma$ is convex (cf. \cite{FA01}). In particular if we denote by $\overline{\bigwedge^k\Gamma^{\mathrm{st}}}$ the convex closure of $\bigwedge^k\Gamma^{\mathrm{st}}$, then $\M{\overline{\bigwedge^k\Gamma^{\mathrm{st}}}}$ is a semi-stable resolution of $\M{{\bigwedge^k\Gamma^{\mathrm{st}}}}$. Moreover $\M{\overline{\bigwedge^k\Gamma^{\mathrm{st}}}}$ is given by a sequence of blow-ups $\M{\overline{\bigwedge^k\Gamma^{\mathrm{st}}}}\rightarrow \Projsp(\bigwedge^k\Lambda_0 ) $ (cf. \cite{FA01}). The candidate $\can$ is now defined as the strict transform $\can\rightarrow\gr{\Lambda_0}$ of the Pl\"ucker embedding $\gr{\Lambda_0}\subseteq\Projsp(\bigwedge^k\Lambda_0)$ under this sequence of blow-ups.\\
 Although we are not able to show the semi-stability of $\can$, we can still show the theorem below under some technical conditions. Let us denote by $\cangenstrpl$ the blow-up of $\M{{\bigwedge^k\Gamma^{\mathrm{st}}}}$ constructed similarly to the blow-up $\cangenstr\rightarrow\Mgr{k}{\Gamma^{\mathrm{st}}}$ (see Section \ref{sec_candidate} for a more detailed discussion).
 \begin{prop*}
 	Assume Conjecture \ref{conj_convexcl_bijection} and that $\cangenstr$ is semi-stable. Then $\cangenstr$ and $\can$ coincide in both of the following cases:
	\begin{enumerate}
		\item $\cangenstrpl$ is semi-stable
		\item for every irreducible component $C$ of $\cangenstrpl_{\res}$ the intersection $C\cap\cangenstr$ is a union of irreducible components
	\end{enumerate}
 \end{prop*} 

We should also remark that it is not hard to show that for $n\leq 4$ all the candidates $\can$, $\cangenstr$ and $\cangen$ are isomorphic and give semi-stable resolutions of $\loc$.

\bigskip
\noindent \textit{Acknowledgements.} First and foremost I want to thank U. G\"ortz who introduced me tho the subject of local models. The patience in his support and guidance contributed a crucial part in my development. I also want to thank T. Richarz for his useful feedback. This work is part of a dissertation at the faculty of mathematics of the university Duisburg-Essen and was partially funded by the DFG Priority Programme SPP 1489:  Algorithmic and Experimental Methods in Algebra, Geometry and Number Theory.

\section{Mustafin Varieties}
\label{sec_mustafin_varieties}
In this chapter we will start by defining Mustafin varieties $\Mgr{k}{\Gamma}$ for two natural numbers $n$ and $k$ and a finite set $\Gamma$ of homothety classes of lattices in $\quot^n$. These schemes first appeared in \cite[Chapter 2]{Mu72} for $k=1$ and $n$ arbitrary and were studied later by Mustafin in \cite[Chapter 2]{Mu78}. The name Mustafin variety was introduced in \cite[Definition 1.1.]{CHSW11} and generalised to $n$ and $k$ arbitrary (and even arbitrary flag types) in \cite[Definition 2.1]{Hae14}. Once we have defined Mustafin varieties the flatness of the local model $\loc$ immediately identifies the local model with the Mustafin variety $\Mgr{k}{\Gamma^{\mathrm{st}}}$ for the standard lattice chain. \\
In the second part of this chapter we will follow \cite{FA01} and describe Mustafin varieties $\M{\Gamma}$ for a convex set of lattice classes by blow-ups of $\Projsp(\Lambda )$ for any class $[\Lambda]$ in $\Gamma$ and show their semi-stability.

\subsection{The local model as a Mustafin Variety}

The definition of Mustafin varieties is based on the rather classical construction of the join of two schemes. The definition can be found for example in \cite[Chapter 2]{Mu72} or \cite[Definition 2.1]{Ha11}.

\begin{definition}
	For two reduced and separated $\rg$-schemes $X_1$ and $X_2$ with identical generic fiber $X_{1,\quot}=X_{2,\quot}$ the \emph{join} $X_1\bigvee X_2$ is defined as the scheme theoretic closure of the generic fiber $X_{1,\quot}=X_{2,\quot}$ diagonally embedded in $X_1\times_{\rg} X_2$.
\end{definition}

\begin{definition} 
	Let $\Gamma$ be a finite set of homothety classes of $\rg$-lattices in $\quot^n$ for a fixed $n\in \N$ and $\Gamma^{\mathrm{rep}}$ be a set of representatives. For a lattice $\Lambda\in\Gamma^{\mathrm{rep}}$ and a fixed $k\in [n]$ the inclusion $\Lambda\subseteq\quot^n$ identifies $\Lambda\otimes_{\rg}\quot\cong \quot^n$ and hence the generic fiber of $\gr{\Lambda}$ is naturally isomorphic to $\gr{\quot^n}$. The \emph{Mustafin variety} is now defined as the join $\Mgr{k}{\Gamma}\defeq \bigvee_{\Lambda\in\Gamma^{\mathrm{rep}}} \gr{\Lambda}$ over $\rg$.
\end{definition}

\begin{rem}
\label{rem_mustafin_for_repr}
	For two representatives $\Lambda$ and $\unif^m\Lambda $ of the same homothety class we get a canonical isomorphism $\gr{\Lambda}\cong\gr{\unif^m\Lambda}$ and hence up to canonical isomorphism the definition is independent of the choice of representatives.
\end{rem}

\begin{rem}
	Using the flatness of the local model $\loc$, shown in \cite{G01}, and the embedding $\loc\subseteq\prod_{[\Lambda]\in\Gamma^{\mathrm{st}}}\gr{\Lambda}$ we observe that the local model agrees with the closure of its generic fiber. Hence the local model coincides with the Mustafin variety $\Mgr{k}{\Gamma^{\mathrm{st}}}$.
\end{rem}

\begin{definition}
	For a finite set $\Gamma$ of $\rg$-lattice classes in $\quot^n$ and a subset $\Gamma'\subseteq\Gamma$ the projection $\prod_{[\Lambda]\in \Gamma}\gr{\Lambda}\rightarrow \prod_{[\Lambda]\in \Gamma'}\gr{\Lambda}$ induces a projection $\Mgr{k}{\Gamma}\rightarrow\Mgr{k}{\Gamma'}$ which we denote by $\muspro{\Gamma}{\Gamma'}$. For the case $\Gamma'=\{[\Lambda]\}$ we will also write $\musprol{\Lambda}$ for $\muspro{\Gamma}{\Gamma'}$ whenever $\Gamma$ is clear from the context.
\end{definition}

These projections will play an important role in the following so let us collect some of their properties. 

\begin{lemma}\cite[Lemma 3.1]{Hae14}
\label{lem_irr_comp_birational_map}
	For finite sets of lattice classes $\Gamma'\subseteq \Gamma$ and an irreducible component $C'$ of $\Mgr{k}{\Gamma'}_{\res}$ there is a unique irreducible component $C$ of $\Mgr{k}{\Gamma}_{\res}$ such that $\pr_{\Gamma,\Gamma'}(C)=C'$. Furthermore the map $\muspro{\Gamma}{\Gamma'}\vert_{C}\colon C\rightarrow C' $ is birational.
\end{lemma}

\begin{definition}\label{def_convex}
	A finite set $\Gamma$ of $\rg$-lattice classes in $\quot^n$ is called \emph{convex} if for any two classes $[\Lambda],[\Lambda']\in\Gamma$ and any representatives $\Lambda$ and $\Lambda'$ the class of the intersection $\Lambda\cap\Lambda'$ is again in $\Gamma$. For an arbitrary finite set of $\rg$-lattice classes $\Gamma$ the intersection of all convex sets containing $\Gamma$ is called the \emph{convex closure} and is denoted by $\overline{\Gamma}$.
\end{definition}
\begin{rem}
	This notion of convexity plays an important role for example in \cite{FA01}. In \cite{SJ07} a reformulation relating it to the notion of \emph{tropical convexity} was given. This reformulation is based on the identification of an apartment of the Bruhat-Tits building with the points $\Z^n/\Z(1,.\dots,1)$ of the tropical projective torus $\R^n/\R(1,.\dots,1)$ (see for example \cite[Chapter 4]{CHSW11}).\\
	In \cite[Chapter 2]{CHSW11} also the relation to the more intrinsic notion of \emph{metrical convexity} is discussed. We call $\Gamma$ metrically convex if for $[\Lambda]$ and $[\Lambda']$ in $\Gamma$ all geodesics for the graph metric of the Bruhat-Tits building are contained in $\Gamma$, i.e. any $[\Lambda'']$ with 
	\begin{align*}
		\mathrm{dist}([\Lambda],[\Lambda''])+\mathrm{dist}([\Lambda''],[\Lambda'])= \mathrm{dist}([\Lambda],[\Lambda'])
	\end{align*}  
	is contained in $\Gamma$. Since this equality is satisfied for $[\Lambda'']$ with $\Lambda''= \unif^{n}\Lambda\cap \unif^{n'}\Lambda'$ we see that metrical convexity implies convexity in the sense of Definition \ref{def_convex}.
\end{rem}

\begin{definition}
	We call two distinct lattices classes $[\Lambda]$ and $[\Lambda']$ neighbours and denote this by $[\Lambda ]\sim[\Lambda']$ if $\{[\Lambda],[\Lambda'] \}$ is convex. This is equivalent to the existence of two representatives $\Lambda$ and $\Lambda'$ with $\unif \Lambda\subseteq\Lambda'\subseteq\Lambda$.
\end{definition}

\begin{rem}
\label{rem_ex_losure_two_lat}
	In general it is hard to compute the convex closure for a finite set $\Gamma$. But for two lattices classes $[\Lambda]$ and $[\Lambda']$ the convex closure $\overline{\{[\Lambda],[\Lambda'] \}}$ can be described in the following way. Fix a representative $\Lambda$ for $[\Lambda]$ and let $\Lambda'$ be the representative of $[\Lambda']$ minimal with $\Lambda\subseteq\Lambda'$. Then this description gives us a chain of neighbouring classes
	\begin{align*}
		[\Lambda ] = [ \Lambda\cap \Lambda' ] \sim [\unif^{1} \Lambda\cap \Lambda' ]\sim\dots \sim [\unif^{k} \Lambda\cap \Lambda' ]=[\Lambda' ]
	\end{align*}
	for $k$ minimal with $[\unif^{k} \Lambda\cap \Lambda' ]=[\Lambda' ]$. The classes $[\unif^{i} \Lambda\cap \Lambda' ]=[\Lambda' ]$ and $[\unif^{k} \Lambda\cap \Lambda' ]=[\Lambda' ]$ for $i< j$ in $[k+1]$ are neighbours if and only if $i+1=j$. We obtain that the convex closure $\overline{\{[\Lambda ],[\Lambda']\}}$ is $\{[\unif^i \Lambda\cap \Lambda' ]\vert i\in \Z  \}$.
\end{rem}

To indicate the importance of convexity let us cite the following lemma.

\begin{lemma}\cite[Lemma 5.8]{CHSW11}
\label{lem_convex_irr_com_correspond_to_lattices}
	For a finite convex set of lattice classes $\Gamma$ and an irreducible component $C$ of $\M{\Gamma}_\res$ there exists a unique class $[\Lambda]$ in $\Gamma$ such that $C$ is the unique irreducible component mapping birationally to $\Projsp(\Lambda )_{\res}$. In particular the number of irreducible components of $\M{\Gamma}_\res$ coincides with the number of lattice classes in $\Gamma$.
\end{lemma}

\subsection{Mustafin varieties as blow-ups}

In this section we focus on the much better understood case of Mustafin varieties of projective spaces. We will prove in Proposition \ref{prop_convex_mustafin_semi-stable} that for a convex set of lattice classes $\Gamma$ the Mustafin variety $\M{\Gamma}$ is semi-stable. This was first proven by Mustafin in \cite[Proposition 2.1]{Mu78} for the case that $\Gamma$ forms a simplex and generalised by Faltings in \cite[Chapter 5]{FA01} to the convex case. For the proof Faltings uses the moduli description of $\M{\Gamma}$ and then easily reduces to the case of a simplex.\\
In contrast to this approach we will analyse the description of $\M{\Gamma}$ as a sequence of blow-ups (cf. \cite[proof of Lemma 5]{FA01}) and will reprove with Proposition \ref{prop_mustafin_as_blow-ups} that for any $\Lambda$ in $\Gamma$ the Mustafin variety is obtained by a sequence of blow-ups starting with $\Projsp(\Lambda)$ in smooth centers. Using  \cite[Lemma 3.2.1]{GEN00} saying that blow-ups preserve semi-stability under some hypothesis, we are able to reprove the semi-stability of $\M{\Gamma}$ for convex sets of lattice classes in this way. Before we prove the claims above, we will start with some technical preparations.

\begin{lemma}\cite[Lemma in Chapter 2]{Mu78}
\label{lem_monoidaltransorm}
	Let $X$ be an integral noetherian regular scheme with smooth subschemes $Y_1, Y_2 \subset X$ such that $Y_1\subseteq Y_2$ or $Y_1 \cap Y_2=\emptyset$, then we can identify the blow-ups $\Bl{\str{Y_2}}{\Bl{Y_1}{X}}$ and $\Bl{\tot{Y_1}}{\Bl{Y_2}{X}}$. Moreover the blow-ups coincide with the join $\Bl{Y_1}{X}\bigvee \Bl{Y_2}{X}$ together with the projections $\pr_1$ and $\pr_2$ to its factors. The situation is summarised in the following commutative diagram:
	\begin{align*}
		\xymatrix{
		\Bl{\str{Y_2}}{\Bl{Y_1}{X}}\ar[d]_{\pr_{\str{Y_2}}}\ar@{=}[r] &\Bl{Y_1}{X}\bigvee \Bl{Y_2}{X}\ar[dl]_{\pr_1}\ar[dr]^{\pr_2} \ar@{=}[r]&\Bl{\tot{Y_1}}{\Bl{Y_2}{X}}\ar[d]^{\pr_{\tot{Y_1}}}\\
		\Bl{Y_1}{X}\ar_{\pr_{Y_1}}[dr]&&\Bl{Y_2}{X}\ar^{\pr_{Y_2}}[dl]\\
		&X
		}
	\end{align*}
\end{lemma}

	Let $\Gamma$ be a finite convex set of $\rg$-lattice classes in $\quot^n$ and $[\Lambda] \in \Gamma$ such that $\Gamma\setminus\{[\Lambda]\}$ is still convex. The projection $\muspro{\Gamma\setminus\{[\Lambda]\}}{\Gamma}$ will turn out to be a blow-up.
	First we define a closed subscheme $Z_{\Gamma,\Lambda}\subseteq \M{\Gamma\setminus \{\Lambda\}}$ that will turn out to be the center of the blow-up.\\
	Fix a representative $\Lambda$ for the homothety class $[\Lambda]$ and for every lattice class $[\Lambda']\in \Gamma\setminus \{[\Lambda]\}$ take the unique representatives $\Lambda'$ such that $\Lambda'\subseteq\Lambda \nsupseteq \unif^{-1}\Lambda'$. Define a closed subscheme $Z_{\Lambda'\subseteq\Lambda}$ of $\Projsp(\Lambda')_{\res}$ endowed with the reduced scheme structure as the complement of the open subscheme where the induced birational map $\Projsp(\Lambda')\dashedrightarrow \Projsp(\Lambda)$ is defined. This closed subscheme can be identified with $\Projsp(V_{\Lambda'\subseteq\Lambda})\subseteq\Projsp(\Lambda')_{\res}$ for the module $V_{\Lambda'\subseteq\Lambda}$ defined as the kernel of the map $\Lambda'_{\res}\rightarrow \Lambda_{\res}$ induced by the inclusion. We obtain the following commutative diagram:
	\begin{align*}
		\xymatrix{
		\musprol{\Lambda'}^{-1}(Z_{\Lambda'\subseteq\Lambda})\ar[d]\ar[r]&\M{\Gamma\setminus\{[\Lambda]\}}\ar^{\musprol{\Lambda'}}[d]\\
		Z_{\Lambda'\subseteq\Lambda}\ar[r]&\Projsp(\Lambda')
		}
	\end{align*}
	Now we take all the inverse images $\musprol{\Lambda'}^{-1}(Z_{\Lambda'\subseteq\Lambda})$ under the natural projections and define 
	\begin{align*}
		Z_{\Gamma,\Lambda}\defeq \bigcap\limits_{[\Lambda']\in \Gamma\setminus\{[\Lambda]\}}\musprol{\Lambda'}^{-1}(Z_{\Lambda'\subseteq\Lambda})\subseteq\M{\Gamma\setminus\{[\Lambda] \}}.
	\end{align*}

\begin{lemma}
\label{lem_interschanging_lambda}
	Let $\Gamma$ be a finite convex set of $\rg$-lattice classes in $\quot^n$ and $[\Lambda]$ in $\Gamma$ such that $\Gamma\setminus\{[\Lambda]\}$ is convex.
	Then we can find for every $[\Lambda']\in \Gamma\setminus\{[\Lambda]\}$ a neighbour $[\Lambda'']\in \Gamma\setminus\{[\Lambda]\}$ of $[\Lambda]$ such that $\musprol{\Lambda''}^{-1}(Z_{\Lambda''\subseteq\Lambda})\subseteq \musprol{\Lambda'}^{-1}(Z_{\Lambda'\subseteq\Lambda})$. Therefore we get 
	\begin{align*}
		Z_{\Gamma,\Lambda}= \bigcap\limits_{\overset{[\Lambda']\in \Gamma\setminus\{[\Lambda]\}}{[\Lambda']\sim [\Lambda]}}\musprol{\Lambda'}^{-1}(Z_{\Lambda'\subseteq\Lambda})\subseteq \M{\Gamma\setminus\{[\Lambda]\}},
	\end{align*}  
	where $[\Lambda']\sim[\Lambda]$ denotes the neighbouring relation defined above.
\end{lemma}

\begin{proof}
	Fix $[\Lambda']\in \Gamma\setminus\{[\Lambda]\}$. In Remark \ref{rem_ex_losure_two_lat} we have explicitly calculated the convex set $\overline{\{[\Lambda'],[\Lambda]\}}$ to be $\left\{[\Lambda']=[\Lambda_0],[\Lambda_1],\dots,[\Lambda_k]=[\Lambda]\right\}$ where in particular $[\Lambda_i]$ is a neighbour of $[\Lambda_{i+1}]$ for $i\leq k-1$. Since $\Gamma$ is convex the closure above is contained in $\Gamma$. Now we fix a representative $\Lambda'$ and take the representatives for $[\Lambda_i]$ such that $\Lambda_i\supseteq\Lambda' \nsubseteq \unif\Lambda_i$ for $0<i\leq k$. We claim that for these representatives we have the inclusions 
	\begin{align*}
		\Lambda'=\Lambda_0\subseteq \Lambda_1\subseteq \dots \subseteq \Lambda_k.
	\end{align*}
	We prove this claim by induction. Since the case $k=1$ is trivial let us assume we have the inclusions above for $k-1$. Now take the representative for $[\Lambda_k]$ such that $\Lambda_k\supseteq \Lambda_{k-1}\supseteq \unif\Lambda_k$. We want to see that this representative has the property $\Lambda_k\supseteq \Lambda'\nsubseteq \unif\Lambda_k$. If we assume that $\unif\Lambda_k\supseteq \Lambda'$, the sequence 
	\begin{align*}
		[\Lambda']=[\Lambda_0\cap \unif\Lambda_k],[\Lambda_1\cap \unif\Lambda_k],\dots,[\Lambda_{k-2}\cap \unif\Lambda_k],[\Lambda_{k-1}\cap \unif\Lambda_k]=[\Lambda]
	\end{align*}
	is a path of smaller length contradicting the minimality. Hence the claim is proven.\\
	Now the inclusion $\Lambda'\subseteq \Lambda$ factors through the representative $\Lambda_{k-1}$ of a neighbour of  $[\Lambda]$ and therefore $V_{\Lambda'\subseteq\Lambda}$ is contained in $V_{\Lambda'\subseteq\Lambda_k}$. The inclusion induces the following commutative diagram
	\begin{align*}
		\xymatrix{
		&\M{\Gamma\setminus\{[\Lambda]\}}\ar[dl]_{\musprol{\Lambda'}}\ar[dr]^{\musprol{\Lambda_k}}\\
		\Projsp(\Lambda')\setminus Z_{\Lambda\subseteq\Lambda'}\ar[dr] \ar[rr]&&\Projsp(\Lambda_k)\setminus Z_{\Lambda\subseteq\Lambda_k} \ar[dl]\\
		&\Projsp(\Lambda)
		}
	\end{align*}
	and therefore we get $\musprol{\Lambda_k}^{-1}(Z_{\Lambda\subseteq\Lambda_k})^c\subseteq \musprol{\Lambda'}^{-1}(Z_{\Lambda\subseteq\Lambda'})^c$. Taking complements and applying $\musprol{\Lambda_k}$ we get $\musprol{\Lambda_k}(\musprol{\Lambda'}^{-1}(Z_{\Lambda'\subseteq\Lambda}))\subseteq Z_{\Lambda_k\subseteq\Lambda}$ and hence the desired result.
\end{proof}

\begin{lemma}
	Fix three different lattice classes $[\Lambda]$, $[\Lambda']$ and $[\Lambda'']$, such that any two of them are neighbours, then $Z_{\{[\Lambda],[\Lambda'],[\Lambda'']\},\Lambda}$ is either $\musprol{\Lambda'}^{-1}(Z_{\Lambda'\subseteq\Lambda})$ or $\musprol{\Lambda''}^{-1}(Z_{\Lambda''\subseteq\Lambda})$.
\end{lemma}

\begin{proof}
	Fix a representative $\Lambda$ for $[\Lambda]$ and take representatives $\Lambda'$ and $\Lambda''$ for the classes $[\Lambda']$ and $[\Lambda'']$ such that $\unif\Lambda\subseteq \Lambda' \subseteq \Lambda$ and $\unif\Lambda\subseteq \Lambda'' \subseteq \Lambda$. Assume without loss of generality that $\Lambda'\subseteq\Lambda''$ and hence $V_{\Lambda'\subseteq\Lambda''}\subseteq V_{\Lambda'\subseteq\Lambda}$. The inclusions induce the diagram
	\begin{align*}
		\xymatrix{
		&\M{\Lambda',\Lambda''}\ar[dl]_{\musprol{\Lambda'}}\ar[dr]^{\musprol{\Lambda''}}\\
		\Projsp(\Lambda')\setminus Z_{\Lambda'\subseteq\Lambda}\ar[dr] \ar[rr]&&\Projsp(\Lambda'')\setminus Z_{\Lambda''\subseteq\Lambda} \ar[dl]\\
		&\Projsp(\Lambda)
		}
	\end{align*}
	and we get $ \musprol{\Lambda'}^{-1}(Z_{\Lambda'\subseteq\Lambda})\subseteq \musprol{\Lambda''}^{-1}(Z_{\Lambda''\subseteq\Lambda}) $. In particular the intersection
	\begin{align*}
		Z_{\{[\Lambda],[\Lambda'],[\Lambda'']\},\Lambda}= \musprol{\Lambda''}^{-1}(Z_{\Lambda''\subseteq\Lambda})\cap \musprol{\Lambda'}^{-1}(Z_{\Lambda'\subseteq\Lambda})
	\end{align*}  
	is simply $\musprol{\Lambda'}^{-1}(Z_{\Lambda'\subseteq\Lambda})$.
\end{proof}

\begin{lemma}
\label{lem_center_are_contained_or_disjoint}
	Fix a convex set of lattice classes $\Gamma$ and two classes $[\Lambda], [\Lambda'] \in \Gamma$ such that $\Gamma\setminus \{[\Lambda]\}$  and $\Gamma\setminus \{[\Lambda']\}$ are still convex. If the set $\{[\Lambda],[\Lambda']\}$ is convex, then one of the two subvarieties $Z_{\Gamma\setminus\{[\Lambda]\},\Lambda'}$ and $Z_{\Gamma\setminus\{[\Lambda']\},\Lambda}$ is contained in the other. In the case that $\{[\Lambda],[\Lambda']\}$ is not convex the two subvarieties $Z_{\Gamma\setminus\{[\Lambda]\},\Lambda'}$ and $Z_{\Gamma\setminus\{[\Lambda']\},\Lambda}$ are disjoint. 
\end{lemma}
\begin{proof}
	First we start with the case where $\{[\Lambda],[\Lambda']\}$ is convex. Fix $[\widetilde\Lambda]\in \Gamma\setminus\{[\Lambda],[\Lambda']\}$ and choose representatives such that $\widetilde\Lambda\subseteq\Lambda,\Lambda' \nsupseteq \unif^{-1}\widetilde\Lambda$. Since $\{[\Lambda],[\Lambda']\}$ is convex, $\Lambda \cap \Lambda'$ is either $\Lambda$ or $\Lambda'$. Let us assume $\Lambda \cap \Lambda'=\Lambda$, then we get $Z_{\widetilde\Lambda\subseteq\Lambda}=Z_{\widetilde\Lambda\subseteq\Lambda\cap \Lambda'}\subseteq Z_{\widetilde\Lambda\subseteq\Lambda'}$ and since $[\widetilde{\Lambda}]$ was arbitrary we get the inclusion $Z_{\Gamma\setminus\{[\Lambda']\},\Lambda}\subseteq Z_{\Gamma\setminus\{[\Lambda]\},\Lambda'}$.\\
	Now let us assume $\{[\Lambda], [\Lambda']\}$ is not convex then since $\Gamma$ is convex we can choose representatives such that $[\Lambda \cap \Lambda']\in \Gamma\setminus\{[\Lambda],[\Lambda']\}$. Now we compute 
	\begin{align*}
		 Z_{\Gamma\setminus\{[\Lambda]\},\Lambda'}\cap Z_{\Gamma\setminus\{[\Lambda']\},\Lambda}= \bigcap\limits_{[\widetilde \Lambda] \in \Gamma \setminus \{[\Lambda],[\Lambda']\}}\pr^{-1}_{\widetilde\Lambda}\left(Z_{\widetilde\Lambda\subseteq \Lambda}\cap Z_{\widetilde\Lambda\subseteq \Lambda'}\right)\subseteq \pr^{-1}_{\Lambda \cap\Lambda'}\left(Z_{\Lambda \cap\Lambda'\subseteq \Lambda}\cap Z_{\Lambda \cap\Lambda'\subseteq \Lambda'}\right),
	\end{align*} but since we have $Z_{\Lambda\cap\Lambda'\subseteq \Lambda}\cap Z_{\Lambda \cap\Lambda'\subseteq \Lambda'}=\emptyset$ we are done.
\end{proof}

Next we want to prove the claims of the beginning of this section by induction. Therefore we first we need to cite the following proposition by Mustafin for the induction start and the lemma below by Faltings to prove the induction step.

\begin{prop}\cite[Proposition 2.1]{Mu78}
\label{prop_inductionstart}
	For two neighbouring lattice classes $[\Lambda]$ and $[\Lambda']$ the projection $\musprol{\Lambda}\colon\M{\{[\Lambda],[\Lambda']\}}\rightarrow \Projsp(\Lambda)$ is the blow-up in the smooth center $Z_{\{\Lambda,\Lambda'\},\Lambda}$. 
\end{prop}

\begin{lemma}\cite[proof of Lemma 5]{FA01}
\label{lem_convexreduction}
	Fix a finite set $\Gamma$ of classes of $\rg$-lattices in $\quot^n$ with at least two elements. For every $[\Lambda]\in\Gamma$ there is a lattice class $[\Lambda']\in \Gamma$ different from $[\Lambda]$ such that $\Gamma\setminus \{[\Lambda']\}$ remains convex. In particular we can find $[\Lambda],[\Lambda']\in \Gamma$ such that $\Gamma\setminus\{[\Lambda]\}$ and $\Gamma\setminus\{[\Lambda']\}$ and hence also $\Gamma\setminus\{[\Lambda],[\Lambda']\}$ remains convex.
\end{lemma}

\begin{rem}
\label{rem_convex_unique_generatingset}
	For every convex set of lattice classes $\Gamma$ there is a unique minimal subset $\Gamma^{\mathrm{gen}}\subseteq\Gamma$ such that the convex closure $\overline{\Gamma^{\mathrm{gen}}}$ is $\Gamma$, cf. \cite[Proposition 5.2.17.]{MS15}. In particular for $[\Lambda]\in \Gamma^{\mathrm{gen}}$ also $\Gamma\setminus\{[\Lambda]\}$ has to be convex and hence $\Gamma^{\mathrm{gen}}=\left\{[\Lambda ]\in\Gamma\middle\vert \Gamma\setminus\{[\Lambda]\} \text{ is convex}\right\}$. 
\end{rem}

\begin{lemma}
\label{lem_main}
	Let $\Gamma$ be a finite convex set of $\rg$-lattice classes in $\quot^n$ with at least two elements and $[\Lambda]$ in $\Gamma$ such that $\Gamma\setminus\{[\Lambda]\}$ is still convex. Then we can describe the map $\pr_{\Gamma,\Lambda}$ by the blow-up $\Bl{Z_{\Gamma,\Lambda}}{\M{ \Gamma\setminus\{[\Lambda]\}}}$ in the smooth center $Z_{\Gamma,\Lambda}$. Furthermore for every $[\Lambda']$ in $\Gamma\setminus\{[\Lambda]\}$ such that $\Gamma\setminus \{[\Lambda']\}$ remains convex and $\Gamma\setminus\{[\Lambda],[\Lambda'] \}$ is not empty the center $Z_{\Gamma,\Lambda}$ is
	\begin{enumerate}[label=(\arabic{enumi})]
		\item \label{lem_main_l_in_ls} the total transform  $\tot{Z_{\Gamma\setminus\{[\Lambda']\}, \Lambda}}$ if $Z_{\Gamma\setminus\{[\Lambda']\},\Lambda}\subseteq Z_{\Gamma\setminus\{[\Lambda]\},\Lambda'}$
		\item \label{lem_main_ls_in_l} the strict transform $\str{Z_{\Gamma\setminus\{[\Lambda']\}, \Lambda}}$ if $Z_{\Gamma\setminus\{[\Lambda]\},\Lambda'}\subseteq Z_{\Gamma\setminus\{[\Lambda']\},\Lambda}$
		\item \label{lem_main_l_dis_ls} the strict or total tramsform $\str{Z_{\Gamma\setminus\{[\Lambda']\}, \Lambda}}=\tot{Z_{\Gamma\setminus\{[\Lambda']\}, \Lambda}}$ if $Z_{\Gamma\setminus\{[\Lambda']\},\Lambda}\cap Z_{\Gamma\setminus\{[\Lambda]\},\Lambda'} = \emptyset$
	\end{enumerate}
	of the blow-up $\M{\Gamma\setminus\{\Lambda\}}\rightarrow \M{\Gamma\setminus\{\Lambda,\Lambda'\}}$.

\end{lemma}
\begin{proof}
	We prove the statement by induction on the number of lattices in $\Gamma$. The base case $ \Gamma=\{[\Lambda_1],[\Lambda_2]\}$ for the induction was shown in Proposition \ref{prop_inductionstart}.\\
	Now for $\sharp \Gamma\geq 3$ fix two lattices $[\Lambda], [\Lambda'] \in \Gamma$ such that $\Gamma\setminus \{[\Lambda]\}$ and $\Gamma\setminus \{[\Lambda']\}$ are still convex. By hypothesis we know that the maps $\M{\Gamma\setminus\{[\Lambda]\}}\rightarrow\M{\Gamma\setminus\{[\Lambda],[\Lambda']\}}$ and $\M{\Gamma\setminus\{[\Lambda']\}}\rightarrow\M{\Gamma\setminus\{[\Lambda],[\Lambda']\}}$ are blow-ups in the center $Z_{\Gamma\setminus\{[\Lambda]\},\Lambda'}$ and $Z_{\Gamma\setminus\{[\Lambda']\},\Lambda}$ respectively. To indicate this we denote the map $\muspro{\Gamma\setminus\{[\Lambda]\}}{\Gamma\setminus\{[\Lambda],[\Lambda']\}}$ by $\pr_{Z_{\Gamma\setminus\{[\Lambda]\},\Lambda'}}$ and similarly we write $\pr_{Z_{\Gamma\setminus\{[\Lambda']\},\Lambda}}$ for $\muspro{\Gamma\setminus\{[\Lambda]\}}{\Gamma\setminus\{[\Lambda],[\Lambda']\}}$.\\
	From Lemma \ref{lem_center_are_contained_or_disjoint} we know that in the case where $Z_{\Gamma\setminus\{[\Lambda ]\} ,\Lambda'}$ and $Z_{\Gamma\setminus\{[\Lambda']\} ,\Lambda}$ are not disjoint one has to be contained in the other and we can prove the statement for every case separately. We will first prove Case
	\ref{lem_main_l_in_ls} and then deduce Case \ref{lem_main_ls_in_l}. 
	The remaining Case \ref{lem_main_l_dis_ls} is proven analogously.
	We begin by proving the second part of the lemma i.e. that $Z_{\Gamma,\Lambda}$ is the total transform of ${Z_{\Gamma\setminus\{[\Lambda']\},\Lambda}}$ under the blow-up $\pr_{Z_{\Gamma\setminus\{[\Lambda]\},\Lambda'}}$. We compute
	\begin{align*}
		 \tot{Z_{\Gamma\setminus\{[\Lambda']\},\Lambda}}=\pr^{-1}_{Z_{\Gamma\setminus\{[\Lambda]\},\Lambda'}}(Z_{\Gamma\setminus\{[\Lambda']\},\Lambda})= \bigcap\limits_{[\widetilde\Lambda]\in \Gamma \setminus\{[\Lambda],[\Lambda']\}} \musprol{\widetilde\Lambda}^{-1}(Z_{\widetilde\Lambda\subseteq\Lambda})
		\supseteq \bigcap\limits_{[\widetilde\Lambda]\in \Gamma \setminus \{[\Lambda]\}} \musprol{\widetilde\Lambda}^{-1}(Z_{\widetilde\Lambda\subseteq\Lambda}) = Z_{\Gamma,\Lambda} 
	\end{align*}
	and we have equality since by the assumption of case \ref{lem_main_l_in_ls} we have the inclusion 
	\begin{align*}
		\tot{Z_{\Gamma\setminus\{[\Lambda']\},\Lambda}}=\pr^{-1}_{Z_{\Gamma\setminus\{[\Lambda]\},\Lambda'}}(Z_{\Gamma\setminus\{[\Lambda']\},\Lambda})\subseteq \pr^{-1}_{Z_{\Gamma\setminus\{[\Lambda']\},\Lambda}}(Z_{\Gamma\setminus\{[\Lambda]\},\Lambda})\subseteq \musprol{\Lambda'}(Z_{\Lambda\subseteq\Lambda'}).
	\end{align*} 	\\
	We just have proven that the center $Z_{\Gamma,\Lambda}$ is the total transform $\tot{Z_{\Gamma\setminus\{[\Lambda']\},\Lambda}}$. To finish the proof we need to do some identifications summarised in the following diagram:   
	\begin{align*}
		\xymatrixcolsep{-0.25in}
		\begin{xy}
		\xymatrix{
			\Bl{\tot{Z_{\Gamma\setminus\{[\Lambda']\},\Lambda}}}{\M{\Gamma\setminus\{[\Lambda]\}}}\ar[ddd]\ar@{=}[rd]&&\Bl{\str{Z_{\Gamma\setminus\{[\Lambda]\},\Lambda'}}}{\M{\Gamma\setminus\{[\Lambda']\}}}\ar[ddd]\\
			&\M{\Gamma\setminus\{[\Lambda]\}}\bigvee\M{\Gamma\setminus\{[\Lambda']\}} \ar@{=}[d] \ar@{=}[ru]\\
			&\M{\Gamma}\ar[dl]\ar[dr]\\
			\M{\Gamma\setminus\{[\Lambda]\}}\ar[dr]_{\pr_{Z_{\Gamma\setminus\{[\Lambda]\},\Lambda'}}}&&\M{\Gamma\setminus\{[\Lambda']\}}\ar[dl]^{\ \ \ \ \ \ \  \pr_{Z_{\Gamma\setminus\{[\Lambda']\},\Lambda}}}\\
			&\M{\Gamma\setminus\{[\Lambda],[\Lambda']\}}
		}	
		\end{xy}
	\end{align*} 
	 With Lemma \ref{lem_monoidaltransorm} we then can identify the two blow-ups $\Bl{\tot{Z_{\Gamma\setminus\{[\Lambda']\},\Lambda}}}{\M{\Gamma\setminus\{[\Lambda]\}}}$ and $\Bl{\str{Z_{\Gamma\setminus\{[\Lambda]\},\Lambda'}}}{\M{\Gamma\setminus\{[\Lambda']\}}}$ with the join $ \M{\Gamma\setminus\{[\Lambda]\}}\bigvee \M{\Gamma\setminus\{[\Lambda']\}}$
	which we identify with $\M{\Gamma} $ using the universal properties of the join construction.\\
	Now the second case where $Z_{\Gamma\setminus\{[\Lambda']\},\Lambda}\subseteq Z_{\Gamma\setminus\{[\Lambda]\},\Lambda'}$ follows easily from the identification in the diagram above by interchanging the roles of $\Lambda$ and $\Lambda'$.\\
	By hypothesis the centers $Z_{\Gamma\setminus\{[\Lambda]\},\Lambda'}$ and $Z_{\Gamma\setminus\{[\Lambda']\},\Lambda}$ are smooth. Hence in the cases \ref{lem_main_l_in_ls} and \ref{lem_main_l_dis_ls}  the center $Z_{\Gamma,\Lambda}=\tot{Z_{\Gamma\setminus\{[\Lambda']\},\Lambda}}$ is clearly smooth and in case \ref{lem_main_ls_in_l} the center $Z_{\Gamma,\Lambda}=\str{Z_{\Gamma\setminus\{[\Lambda']\},\Lambda}}$  is a blow-up of a smooth scheme over a field in a smooth center and in particular it is smooth.\\
\end{proof}

\begin{prop}\cite[proof of Lemma 5]{FA01}
\label{prop_mustafin_as_blow-ups}
	Let $\Gamma$ be a finite convex set of $\rg$-lattices in $\quot^n$ and $\Lambda$ representing a class in $\Gamma$. Then  the Mustafin variety $\M{\Gamma}$ is a successive blow-up of $\Projsp(\Lambda)$ in smooth centers. 
\end{prop}

\begin{proof}
	Proving the statement by induction on the number of classes in $\Gamma$, we can assume the statement is true for a convex set of lattice classes $\Gamma$ with $i$ elements. Now for a convex set of lattice classes $\Gamma$ with $i+1$ elements and a class $[\Lambda]$ in $\Gamma$ we use Lemma \ref{lem_convexreduction} to find a lattice class $[\Lambda']\neq [\Lambda]$ such that $\Gamma\setminus\{[\Lambda']\}$ is still convex. Now we are done since $\M{\Gamma\setminus\{[\Lambda']\}}$ is obtained by a sequence of blow-ups of $\Projsp(\Lambda )$ in smooth centers by assumption and the map $\M{\Gamma}\rightarrow \M{\Gamma\setminus\{[\Lambda']\}}$ is a blow-up in a smooth center by Lemma \ref{lem_main}.
\end{proof}

\begin{rem}
\label{rem_intersection_neigh}
	For two classes $[\Lambda]$ and $[\Lambda']$ in a convex set $\Gamma$, the two irreducible components $C_{\Lambda}$ and $C_{\Lambda '}$ of $\M{\Gamma }_{\res}$ are the exceptional divisors of the blow-ups $\pr_{Z_{\Gamma,\Lambda}}$ and $\pr_{Z_{\Gamma,\Lambda'}}$ respectively.
	Using Lemma \ref{lem_center_are_contained_or_disjoint} it is easy to see, that $C_{\Lambda}$ and $C_{\Lambda '}$ are disjoint if $[\Lambda]$ and $[\Lambda']$ are not neighbours. Also the converse is true and a proof can be found in \cite[Theorem 2.10]{CHSW11}.
\end{rem}

We are now prepared to use the description of Mustafin varieties as a sequence of blow-ups and the following lemma on semi-stability under blow-ups to get a new proof of the semi-stability of $\M{\Gamma}$ for $\Gamma$ convex (cf. \cite[Chapter 5]{FA01}). 

\begin{lemma}\cite[Lemma 3.2.1]{GEN00}
\label{lem_gen_semi-stable}
	Let $X$ be semi-stable $\rg$-scheme and $Y\subseteq X_{\res}$ a closed subscheme of the special fiber. Suppose $Y$ is smooth over $\res$ and the intersection $Y\cap \sing{X_{\res}}$ is a simple normal crossing divisor on $Y$. Then the blow-up $\bl{Y}{X}$ is semi-stable.
\end{lemma}

\begin{prop}
\label{prop_convex_mustafin_semi-stable}
		For a finite convex set $\Gamma$ of $\rg$-lattices in $\quot^n$ the Mustafin variety $\M{\Gamma}$ is semi-stable.
\end{prop}

\begin{proof}
	Fix a class $[\Lambda]$ in $\Gamma$. We will prove by induction on the number of elements of $\Gamma$ that the center $Z_{\Gamma,\Lambda}$ in $\M{\Gamma\setminus\{[\Lambda]\}}$ satisfies the conditions of Lemma \ref{lem_gen_semi-stable} and hence $\M{\Gamma}$ is semi-stable. Since we already know that $Z_{\Gamma,\Lambda}$ is smooth, we are left to show that\linebreak $Z_{\Gamma,\Lambda}\cap \sing{\M{\Gamma\setminus\{[\Lambda]\}}_{\res}}$ is a simple normal crossing divisor on $Z_{\Gamma,\Lambda}$.\\
	 Assume that we can find a class $[\Lambda']$ different from $[\Lambda]$ such that $\Gamma\setminus\{[\Lambda']\}$ is still convex and $[\Lambda']$ is not a neighbour of $[\Lambda]$. Then $Z_{\Gamma\setminus\{[\Lambda]\},\Lambda'}$ and $Z_{\Gamma\setminus\{[\Lambda']\},\Lambda}$ are disjoint and hence the blow-up $\M{\Gamma\setminus\{[\Lambda]\}}\rightarrow \M{\Gamma\setminus\{[\Lambda],[\Lambda']\}}$ restricts to an isomorphism in a neighbourhood of $Z_{\Gamma, \Lambda}\cong Z_{\Gamma\setminus\{[\Lambda']\},\Lambda}$ and the statement follows directly from the induction hypothesis.\\
	Now take any $[\Lambda']$ different from $[\Lambda]$ such that $\Gamma\setminus\{[\Lambda']\}$ is still convex. Using the previous step we can assume that $[\Lambda]$ and $[\Lambda']$ are neighbours. Then we know by Lemma \ref{lem_main} that one of $Z_{\Gamma\setminus\{[\Lambda']\} , \Lambda}$ and $Z_{\Gamma\setminus\{[\Lambda]\} , \Lambda'}$ contains the other. Let us assume that we are in the situation $Z_{\Gamma\setminus\{[\Lambda]\} , \Lambda'}\subseteq Z_{\Gamma\setminus\{[\Lambda']\} , \Lambda}$. The other case is proven analogously. In particular by Case \ref{lem_main_ls_in_l} of Lemma \ref{lem_main} this implies that $Z_{\Gamma , \Lambda}$ is the strict transform of $Z_{\Gamma\setminus\{[\Lambda']\} , \Lambda}$.
	By induction hypothesis we know that we can choose an \'{e}tale local presentation for $\M{\Gamma\setminus\{[\Lambda]\}}$ of the form 
	\begin{align*}
		\rg[x_0,\dots,x_m]/ \left(\prod\limits_{i\leq r}x_i -\unif\right) .
	\end{align*}   
	Again by induction hypothesis we know that both $Z_{\Gamma\setminus \{[\Lambda']\} , \Lambda}\cap  \sing{\M{\Gamma\setminus\{[\Lambda],[\Lambda'] \}}_{\res}}$ and $Z_{\Gamma\setminus \{[\Lambda]\} , \Lambda'}\cap  \sing{\M{\Gamma\setminus\{[\Lambda],[\Lambda'] \}}_{\res}}$ are simple normal crossing divisors on $Z_{\Gamma\setminus \{[\Lambda']\},\Lambda}$ and $Z_{\Gamma\setminus \{[\Lambda]\},\Lambda'}$ respectively. In particular as shown in \cite[proof of Lemma 3.2.1]{GEN00} we choose an \'{e}tale local presentation such that $Z_{\Gamma\setminus \{[\Lambda]'\},\Lambda}$ is of the form $V(x_r,\dots ,x_{m_1})$ and\linebreak $Z_{\Gamma\setminus \{[\Lambda]\},\Lambda'}= V(x_r,\dots ,x_{m_2})$ for some $m_1\leq m_2\leq m$. Now the strict transform $Z_{\Gamma,\Lambda}$ is cut out by the equations $V(X_r,\dots ,X_{m_1})$ in the $\Gm$-quotient of 
	\begin{align*}
		V(X_r,\dots ,X_{m_2})^c\subseteq \spec{\rg[\lambda, x_1,\dots,x_{r-1}, X_r,\dots X_{m_2},x_{m_2+1},\dots,x_m]/\left(\left(\prod\limits_{i\leq r-1} x_i\right) X_r\lambda-\unif \right)}
	\end{align*} 
	describing the blow-up.
	The intersection of $Z_{\Gamma,\Lambda}$ with the singular locus is now cut out of $Z_{\Gamma,\Lambda}$ by the product of the regular sequence $  x_1,\dots,x_{r-1}, \lambda $. This clearly defines a simple normal crossing divisor before taking the $\Gm$-quotient. But $  x_1,\dots,x_{r-1}, \lambda $ also define regular sequences in the charts of the $\Gm$-quotient and hence $Z_{\Gamma,\Lambda}\cap \sing{\M{\Gamma\setminus\{[\Lambda]\}}_{\res}}$ is a simple normal crossing divisor on $Z_{\Gamma,\Lambda}$.\\ 
	In summary $Z_{\Gamma,\Lambda}$ again satisfies the condition of Lemma \ref{lem_gen_semi-stable}.
\end{proof}

To end this section let us prove a useful lemma on the behaviour of irreducible components under the blow-ups of Lemma \ref{lem_main}.

\begin{lemma}
\label{lem_inverse_of_center}
	Fix a convex set of lattice classes $\Gamma$ and a lattice $\Lambda$ in $\Gamma$ such that $\Gamma\setminus \{[\Lambda]\}$ is again convex. For the projection $\pr\colon \M{\Gamma}\rightarrow \M{\Gamma\setminus \{[\Lambda]\}}$ the inverse image $\pr^{-1}(\pr(C))$ of the image of an irreducible component $C$ of $\M{\Gamma}_{\res}$ is a union of irreducible components.
\end{lemma}

\begin{proof}
	From Lemma \ref{lem_convex_irr_com_correspond_to_lattices} we get for every class $[\Lambda']$ in $\Gamma$ a unique irreducible component $C_{\Lambda'}$ of $\M{\Gamma}_{\res}$ surjecting to $\Projsp(\Lambda' )_{\res}$ under the natural projection. Note that for an irreducible component $C'$ of $\M{\Gamma\setminus\{[\Lambda]\}}_{\res}$ the strict transform under the blow-up $\pr\colon \M{\Gamma}\rightarrow \M{\Gamma\setminus \{[\Lambda]\}}$ is again an irreducible component. Hence the irreducible component $C_\Lambda$ of $\M{\Gamma}_\res$ coincides with the exceptional divisor of $\pr$.\\	
	Let us prove the statement by induction on the number of elements in $\Gamma$. Fix $[\Lambda']$ in $\Gamma$ different from $[\Lambda]$. If the classes $[\Lambda]$ and $[\Lambda']$ are not neighbours, then $C_\Lambda$ and $C_{\Lambda'}$ are disjoint by Remark \ref{rem_intersection_neigh} and hence $\pr(C_{\Lambda'} )$ is disjoint to the center of the blow-up $\pr$. In particular $\pr^{-1}(\pr(C_{\Lambda'} ))$ is $C_{\Lambda'}$.\\
	Let us assume from now on that $[\Lambda]$ and $[\Lambda']$ are neighbours. Assume there is a class $[\Lambda'']$ in $\Gamma$ not neighbouring $[\Lambda']$ and such that $\Gamma\setminus\{[\Lambda'' ]\}$ is still convex. We fix notation as in the following diagram
	\begin{align*}
	\xymatrix{
		&\M{\Gamma}\ar[dr]^{\pr_{\Lambda''}}\ar[dl]_{\pr_{\Lambda}} \\
		\M{\Gamma\setminus\{[\Lambda]\}}\ar[dr]_{\pr_{\Lambda,\Lambda''}}&&\M{\Gamma\setminus\{[\Lambda'']\}}\ar[dl]^{\pr_{\Lambda'',\Lambda}}\\
		&\M{\Gamma\setminus\{[\Lambda],[\Lambda''] \}}
	}
	\end{align*}
	The blow-up $\pr_{\Lambda,\Lambda''}$ restricts to an isomorphism on $\pr_{\Lambda}(C_{\Lambda'} )$ and similarly $\pr_{\Lambda''}$ restricts to an isomorphism on $C_{\Lambda'}$. Now the image $\pr_{\Lambda''}(C_{\Lambda'})$ is again an irreducible component and by induction hypothesis we know that $\pr_{\Lambda'',\Lambda}^{-1}(\pr_{\Lambda'',\Lambda}\circ \pr_{\Lambda''}(C_{\Lambda'}))$ is either $\pr_{\Lambda''}(C_{\Lambda'})\cup\pr_{\Lambda''}(C_\Lambda)$ or $\pr_{\Lambda''}(C_{\Lambda'})$. In particular  $\pr_{\Lambda}^{-1}(\pr_{\Lambda}(C_{\Lambda'}))$ is either $C_{\Lambda'}$ or $C_{\Lambda'}\cup C_{\Lambda}$.\\	
	Assume there is a class $[\Lambda'']$ in $\Gamma$ not neighbouring $[\Lambda]$ and such that $\Gamma\setminus\{[\Lambda'' ]\}$ is still convex. Using the previous step we can assume that $[\Lambda'' ]$ is a neighbour of $[\Lambda' ]$. Now $\pr_{\Lambda''}$ restricts to an isomorphism on $C_{\Lambda} $ and similarly $\pr_{\Lambda,\Lambda''}$ restricts to an isomorphism on $\pr_{\Lambda}(C_\Lambda)$.
	We now compute 
	\begin{align*}
		\pr_{\Lambda}^{-1}\left(\pr_{\Lambda}\left(C_{\Lambda'}\right)\right)\cap C_{\Lambda}& 
		= \pr_{\Lambda}^{-1}\left(\pr_{\Lambda}\left(C_{\Lambda'}\cap C_{\Lambda} \right)\right)
		=  (\pr_{\Lambda,\Lambda''}\circ \pr_{\Lambda})^{-1}\left((\pr_{\Lambda,\Lambda''}\circ \pr_{\Lambda})\left(C_{\Lambda'}\cap C_{\Lambda}\right)\right)\\
		&= (\pr_{\Lambda'',\Lambda}\circ \pr_{\Lambda''})^{-1}\left((\pr_{\Lambda'',\Lambda}\circ \pr_{\Lambda''})\left(C_{\Lambda'}\cap C_{\Lambda}\right)\right)\\
		&=\pr_{\Lambda''}^{-1}\left(\pr_{\Lambda'',\Lambda}^{-1}(\pr_{\Lambda'',\Lambda}\circ \pr_{\Lambda''}(C_{\Lambda'}))\cap \pr_{\Lambda''}(C_{\Lambda}) \right)
	\end{align*}
	and again by induction hypothesis we know that $\pr_{\Lambda'',\Lambda}^{-1}(\pr_{\Lambda'',\Lambda}\circ \pr_{\Lambda''}(C_{\Lambda'}))$ is either $\pr_{\Lambda''}(C_{\Lambda'})$ or $\pr_{\Lambda''}(C_{\Lambda'})\cup\pr_{\Lambda''}(C_\Lambda)$. We conclude that for $\pr_{\Lambda}^{-1}\left(\pr_{\Lambda}\left(C_{\Lambda'}\right)\right)= C_{\Lambda'}\cup \left( \pr_{\Lambda}^{-1}\left(\pr_{\Lambda}\left(C_{\Lambda'}\right)\right)\cap C_{\Lambda} \right)$ we get 
	 \begin{align*}
	 \pr_{\Lambda}^{-1}\left(\pr_{\Lambda}\left(C_{\Lambda'}\right)\right)=
	  \begin{cases}
	 	C_{\Lambda'}& \text{ if }\pr_{\Lambda'',\Lambda}^{-1}(\pr_{\Lambda'',\Lambda}\circ \pr_{\Lambda''}(C_{\Lambda'}))= \pr_{\Lambda''}(C_{\Lambda'})\\
	 	C_{\Lambda'}\cup C_{\Lambda} & \text{ if } \pr_{\Lambda'',\Lambda}^{-1}(\pr_{\Lambda'',\Lambda}\circ \pr_{\Lambda''}(C_{\Lambda'}))= \pr_{\Lambda''}(C_{\Lambda'})\cup\pr_{\Lambda''}(C_\Lambda).
	 \end{cases}
	 \end{align*}	   
	Now we assume that there are two distinct classes $[\Lambda'' ]$ and $[\Lambda''']$ such that $\Gamma\setminus\{[\Lambda'' ]\}$ and $\Gamma\setminus\{[\Lambda''']\}$ are convex and $[\Lambda'' ]$ and $[\Lambda''']$ are not neighbours. By the reductions above we can further assume that $[\Lambda'' ]$ and $[\Lambda''']$ are are different from $[\Lambda]$ and $[\Lambda']$. Let us again fix some notation and extend the diagram above
	\begin{align*}
	\xymatrix{
		&\M{\Gamma}\ar[d]_{\pr_{\Lambda}}\ar[dl]_{\pr_{\Lambda''}}\ar[dr]^{\pr_{ \Lambda'''}}\\
		\M{\Gamma\setminus\{[\Lambda'']\}}\ar[d]_{\pr_{\Lambda'', \Lambda}}] &\M{\Gamma\setminus\{[\Lambda]\}}\ar[dr]^{\pr_{\Lambda, \Lambda'''}}\ar[dl]_{\pr_{\Lambda, \Lambda''}} & \M{\Gamma\setminus\{[\Lambda''']\}}\ar[d]^{\pr_{\Lambda''', \Lambda}} \\
		\M{\Gamma\setminus\{[\Lambda],[\Lambda'' ]\}}&&\M{\Gamma\setminus\{[\Lambda], [\Lambda''']\}}
		}
	\end{align*}
	Then $C_{\Lambda''}$ and $C_{\Lambda'''}$ are disjoint. In particular for every irreducible component $C$ of $\M{\Gamma}_{\res}$ we get 
	\begin{align*}
		C\subseteq \pr_{\Lambda''}^{-1}(\pr_{\Lambda''}(C))\cap \pr_{\Lambda'''}^{-1}(\pr_{\Lambda'''}(C))\subseteq (C\cup C_{\Lambda''})\cap (C\cup C_{\Lambda'''})= C
	\end{align*} 
	and similarly for an irreducible component of $\M{\Gamma\setminus\{[\Lambda]\}}_\res$ and the projections $\pr_{\Lambda,\Lambda''}$ and $\pr_{\Lambda,\Lambda'''}$. Hence $\pr_\Lambda^{-1}(\pr_\Lambda(C_{\Lambda'}))$ is the intersection of the inverse images of the images of $C_{\Lambda'}$ under the projections $\pr_{\Lambda,\Lambda''}\circ \pr_{\Lambda}$ and $\pr_{\Lambda,\Lambda'''}\circ \pr_{\Lambda}$. Now using the induction hypothesis we know that the inverse image of the image of $\pr_{\Lambda''}(C_{\Lambda'})$ under $\pr_{\Lambda'',\Lambda}$ is either $\pr_{\Lambda''}(C_{\Lambda'})$ or $\pr_{\Lambda''}(C_{\Lambda'})\cup \pr_{\Lambda''}(C_{\Lambda})$ and similarly for $\Lambda'''$. Together we get the inverse image of the image of $C_{\Lambda'}$ under $\pr_\Lambda$ is either $C_{\Lambda'}$ or $C_{\Lambda'}\cup C_\Lambda$.\\
	For the remaining case we recall from Remark \ref{rem_convex_unique_generatingset} that $\Gamma$ is the convex closure of the set of $[\Lambda'' ]\in\Gamma$ such that $\Gamma\setminus\{[\Lambda'']\}$ is convex. By the previous steps we can assume that the set $\left\{[\Lambda '']\in\Gamma\middle\vert \Gamma\setminus\{[\Lambda'']\} \text{ is convex} \right\}$ is contained in a simplex. Hence $\Gamma$ is contained in a simplex and we refer to the explicit calculations of the blow-up in \cite[proof of Proposition 2.1]{Mu78} .
\end{proof}

\section{The Pl\"ucker embedding for $\Mgr{k}{\Gamma^{\mathrm{st}}}$}
\label{sec_pl_emb_for_mus}

In this chapter we fix two integers $n\in \N$ and $0\neq k\in [n]$. As in the introduction we denote by $\Gamma^{\mathrm{st}}$ the standard lattice chain in $\quot^n$ and try to get a relation between the irreducible components of the special fibres of the two Mustafin varieties $\Mgr{k}{\Gamma^{\mathrm{st}}}$ and $\M{\bigwedge^k\Gamma^{\mathrm{st}}}$.
In general for any finite set of lattice classes $\Gamma$ the image in $\M{\bigwedge^k\Gamma}_{\res}$ of an irreducible component $C^{\mathrm{gr}}$ of the special fiber of the Mustafin variety $\Mgr{k}{\Gamma}_{\res}$ is again irreducible and hence lies in some irreducible component $C^{\mathrm{pr}}$ of $\M{\bigwedge^k\Gamma}_{\res}$.
For the standard lattice chain $\Gamma^{\mathrm{st}}$ with $k=2$ or $n\leq 5$ and conjecturally for all $n$ this component $C^{\mathrm{pr}}$ is unique. On the other hand we show that in these cases every irreducible component of $\M{\bigwedge^k\Gamma^{\mathrm{st}}}_{\res}$ arises in this way and hence we get a bijective correspondence of irreducible components of the two Mustafin varieties.

\subsection{Irreducible components and linear subspaces}
It is well known that the irreducible components of $\Mgr{k}{\Gamma^{\mathrm{st}}}_\res$ can be indexed by the subsets $I\subseteq [n]$ with $k$ elements. We denote the set of those subsets in the following with $\sset$.
As a first step we will define certain linear subspaces $\Projsp(V_I)$ indexed by $I\in\binom{[n]}{k}$. For these subspaces, we will prove that for every $I$ there is a unique irreducible component of $\M{\overline{\bigwedge^k\Gamma^{\mathrm{st}}}}_\res$ surjecting to $\Projsp(V_I)$ under the projection  $\projectbar{0}\colon \M{\overline{\bigwedge^k\Gamma^{\mathrm{st}}}}\rightarrow \Projsp(\bigwedge^k\Lambda_0 )$ and every irreducible component is obtained in that way. In particular $\Mgr{k}{\Gamma^{\mathrm{st}}}_{\res}$ and $\M{\overline{\bigwedge^k\Gamma^{\mathrm{st}}}}_{\res}$ have the same number of irreducible components.\\
	For the rest of this chapter we fix the basis $ \{e_I \}_{I\in \binom{[n]}{k}}$ of $\bigwedge^k \quot^n$ where $e_I= e_{i_0}\wedge\dots\wedge e_{i_{k-1}}$ for every $I=\{i_0,\dots,i_{k-1}\}\in \binom{[n]}{k}$ and $\{e_i\}$ is the standard basis $\quot^n$.  We define a partial order on this basis by setting $\left\{i_0<\dots<i_{k-1}\right\}\leq \left\{j_0<\dots<j_{k-1}\right\}$ if $i_t\leq j_t$ for all $t\in [k]$.
 \begin{definition}\cite[Chapter XIV 3]{HP94}
 	For $I\in \binom{[n]}{k}$ we write $V_I$ for the subspace of $\bigwedge^k \Lambda_{0,\res}$ generated by $\left\{ e_J\vert J\leq I\right\}$.
 \end{definition}

These subspaces are related to the classical theory of Schubert varieties and we recall their definition.

\begin{definition}
	For $I=\{i_0<\dots <i_{k-1}\}\in \binom{[n]}{k}$ the \emph{Schubert variety} $X_I$ of $\gr{\Lambda_0}_\res$ is the reduced subvariety on the set of subspace $W\in \gr{\Lambda_0}$ such that for all $t\in [k]$ we have $\dim{W\cap \langle e_0,\dots, e_{i_t} \rangle}=t$ and $\dim{W\cap \langle e_0,\dots, e_{i_l}}<t$ for all $l< i_t$.
\end{definition}

The vector spaces defined above were constructed to describe the Schubert varieties of $\gr{\Lambda_0}_\res$ under the Pl\"ucker embedding. 

\begin{lemma}\cite[Chapter XIV 3]{HP94} 
	\label{lem_sub_var_under_pl}
	Using the Pl\"ucker embedding $\gr{\Lambda_0}_{\res}\rightarrow \Projsp(\bigwedge^k \Lambda_0 )_{\res}$ we identify the Schubert variety $X_I$ in $\gr{\Lambda_0}_{\res}$  for $I\in \binom{[n]}{k}$ with the intersection of $\Projsp(V_I)$ with $\gr{\Lambda_0}_{\res}$.
\end{lemma}

\begin{rem}
Let us recall from Lemma \ref{lem_convex_irr_com_correspond_to_lattices} that for an irreducible component $C$ in $\M{\overline{\bigwedge^k \Gamma^{\mathrm{st}}}}_{\res}$ there is a unique lattice $\Lambda_C$ in $\overline{\bigwedge^k\Gamma^{\mathrm{st}}}$ such that under the projection
	\begin{align*}
		\M{\overline{\bigwedge^k\Gamma^{\mathrm{st}}}}\subseteq\prod\limits_{\Lambda\in \bigwedge^k\Gamma^{\mathrm{st}}} \Projsp(\Lambda)\longrightarrow\Projsp(\Lambda_C)
	\end{align*}
	$C$ surjects onto $\Projsp(\Lambda_C)_{\res}$. And conversely for every lattice $\Lambda$ in $\overline{\Gamma^{\mathrm{st}}}$ there is a unique irreducible component $C$ of $\M{\overline{\bigwedge^k\Gamma^{\mathrm{st}}}}_{\res}$ surjecting to $\Projsp(\Lambda )$.\\
	For $[\Lambda]\in \overline{\bigwedge^k\Gamma^{\mathrm{st}}}$ we now choose the unique maximal representative of $[\Lambda]$ contained in $\bigwedge^k \Lambda_0$. The inclusion $\Lambda\subseteq\bigwedge^k\Lambda_0$ gives us a birational map $\Projsp(\Lambda)\dashrightarrow \Projsp(\bigwedge^k\Lambda_0 )$. Note that for another representative $\Lambda'$ of $[\Lambda]$ we can use the identification $\Lambda'=\unif^r\Lambda$ for some $r\in\Z$ to precompose the birational map with the induced isomorphism $\Projsp(\Lambda' )\cong\Projsp(\Lambda )$ and to get a birational map with the same image. We now can identify the special fibers of the images of these birational maps with the images of irreducible components of $\M{\overline{\bigwedge^k \Gamma^{\mathrm{st}}}}_{\res}$.
\end{rem}	
\begin{lemma}
\label{lem_irr_comp_and_im_of_lat}
	Using the notation above we have an equality of sets
	\begin{align*}
		\left\{\projectbar{0}(C)\middle\vert C \text{ irr. component in } \M{\overline{\bigwedge^k\Gamma^{\mathrm{st}}}}_{\res} \right\}=& \left\{\im{\Projsp (\Lambda)_{\res}\dashrightarrow \Projsp(\bigwedge^k \Lambda_0)_{\res}}\middle\vert [\Lambda]\in \overline{\bigwedge^k\Gamma^{\mathrm{st}}} \right\}\\
		\projectbar{0}(C)=  &  \im{\Projsp (\Lambda_C)\dashrightarrow \Projsp(\bigwedge^k \Lambda_0)}_{\res} .
	\end{align*}	
\end{lemma}

\begin{proof}
	Recall from Lemma \ref{lem_irr_comp_birational_map} that for a class $[\Lambda]$ in $\overline{\Gamma^{\mathrm{st}}}$ the projection $\M{\overline{\Gamma^{\mathrm{st}}}}\rightarrow \Projsp(\Lambda)$ restricts to a birational morphism $C_\Lambda \rightarrow \Projsp(\Lambda)$. This implies that the two images $\projectbar{0}(C)$ and $\im{\Projsp (\Lambda_C)\dashrightarrow \Projsp(\bigwedge^k \Lambda_0)}_{\res}$ coincides. Using the bijection between irreducible components of $\M{\overline{\Gamma^{\mathrm{st}}}}_{\res}$ and lattices in $\overline{\Gamma^{\mathrm{st}}}$ discussed in the remark above, the result follows.
\end{proof}

\begin{rem}
	For two $\rg$-lattices $\Lambda$, $\Lambda'$ in $\bigwedge^k\quot^{n}$ and $\Lambda$ maximal in the class $[\Lambda]$ with $\Lambda\subseteq\Lambda'$\linebreak the induced birational map $\Projsp(\Lambda)_{\res}\dasharrow \Projsp(\Lambda')_{\res}$ is defined away from the linear subspace\linebreak $\Projsp(\ker(\Lambda_{\res}\rightarrow \Lambda'_{\res}))\subseteq \Projsp(\Lambda )_{\res}$. In particular the image $\im{\Projsp(\Lambda)\dasharrow \Projsp(\Lambda)}_{\res}$ can be computed as $\Projsp(\im{\Lambda_{\res} \rightarrow \Lambda'_{\res}})$. Therefore we will focus in the following on understanding the submodules $\im{\Lambda_{\res} \rightarrow \Lambda'_{\res}}$ for a representative $\Lambda$ of the class $[\Lambda]$ maximal with $\Lambda\subseteq\Lambda'$ instead of the subvarieties $\im{\Projsp(\Lambda)\dasharrow \Projsp(\Lambda')}_{\res}$.
\end{rem}

\begin{lemma}
\label{lem_images_are_linear_subsp}
	Fix a representative $\Lambda$ of a class in $\overline{\bigwedge^k\Gamma^{st}}$ maximal in its class with $\Lambda\subseteq \bigwedge^k \Lambda_{0}$. Then the image $\im{\Lambda_{\res}\rightarrow \bigwedge^k \Lambda_{0,\res}}$ is of the form $V_I$ for some $I\in \binom{[n]}{k}$ and every $V_I$ for $I\in \binom{[n]}{k}$ arises in this way.
\end{lemma}

\begin{proof}
	Let us start by determining the images of $\bigwedge^k\Lambda_i$ for $\Lambda_i$ in $\Gamma^{\mathrm{st}}$. 
	For the lattice $\unif^l\bigwedge^k \Lambda_i$ with $l\in \Z$ we get
	\begin{align*}
		\unif^l\bigwedge^k \Lambda_{i}= \langle \unif^{m^{i}_I}e_I \rangle_{I\in \binom{[n]}{k} } \text{ with } m^{i}_I= l-\sharp(I\cap [i])\text{ for }I\in \binom{[n]}{k}
	\end{align*}
	and hence 
	\begin{align*}
		\unif^l\bigwedge^k \Lambda_{i}\cap \bigwedge^k \Lambda_{0}= \langle \unif^{m^{i,l}_I}e_I \rangle_{I\in \binom{[n]}{k} } \text{ with } m^{i,l}_I= \max\{l-\sharp(I\cap [i]),0\}\text{ for }I\in \binom{[n]}{k}
	\end{align*}
	and hence we can determine the image as
	\begin{align*}
		\im{(\unif^l\bigwedge^k \Lambda_{i}\cap \bigwedge^k \Lambda_{0})_{\res}\rightarrow \bigwedge^k \Lambda_{0,\res}}=\langle e_I\vert \sharp(I\cap [i])-l\geq0 \rangle.
	\end{align*}
	Now if $l> \min\{k,i\}$ no $I$ satisfies the condition for $e_J$ to appear and the image is trivial. Hence let us assume that $l\leq \min\{k,i\}$. Defining 
	\begin{align*}
		I_i^l\defeq \begin{cases}
			\left\{i-l,\dots, i-1 \right\}\cup \left\{ n-1-k+l ,\dots,n-1\right\} & \text{if }i+k<n+l\\
			\left\{ n-1-k ,\dots,n-1\right\} & \text{otherwise}
		\end{cases} 
	\end{align*}
	gives the reformulation of the condition $\sharp(I\cap [i])\geq l$ to the condition $
	  I\leq I_i^l$.
	Altogether we conclude $\im{(\unif^l\bigwedge^k \Lambda_{i}\cap \bigwedge^k \Lambda_{0})_{\res}\rightarrow \bigwedge^k \Lambda_{0,\res}} = V_{I^l_i}$.\\
	Since every lattice $\Lambda$ representing a class in $\overline{\bigwedge^k\Gamma^{st}}$ is the intersection of lattices of the form as above we need to understand the behaviour of the images under intersections.
	Consider two representatives $\Lambda$, $\Lambda'$ of classes in  $\overline{\bigwedge^k\Gamma^{st}}$ with images
	\begin{align*}
		\im{\Lambda_{\res}\rightarrow \bigwedge^k \Lambda_{0,\res}} = V_{I(\Lambda)}\\
		\im{\Lambda'_{\res}\rightarrow \bigwedge^k \Lambda_{0,\res}} = V_{I(\Lambda')}
	\end{align*} 
	for two subsets ${I(\Lambda)},{I(\Lambda')}\in \binom{[n]}{k}$.
	Now we set 
	\begin{align*}
		\min\{I(\Lambda),I(\Lambda')\}\defeq \{\min\{i_0,i'_0\}<\dots < \min\{i_{k-1},i'_{k-1}\}\}
	\end{align*}  
	with $I(\Lambda)=\{i_0<\dots <i_{k-1}\}$ and $I(\Lambda')=\{i'_0<\dots <i'_{k-1}\}$ and get
	\begin{align*}
		\im{(\Lambda\cap \Lambda')_{\res} \rightarrow \bigwedge^k \Lambda_{0,\res}} = V_{I(\Lambda )}\cap V_{I(\Lambda' )}= V_{\min\{I(\Lambda),I(\Lambda')\}}.
	\end{align*}
	Hence for all lattices $\Lambda$ representing a class in $\overline{\bigwedge^k\Gamma^{st}}$ we can find a set $I(\Lambda)\in \binom{[n]}{k}$ such that $\im{\Lambda_{\res}\rightarrow \bigwedge^k \Lambda_{0,\res}}$ coincides with $V_{I(\Lambda)}$.\\
	The converse will follow directly from the following lemma.  
\end{proof}

\begin{lemma}
\label{lem_every_subset_is_a_minimum}
 	Every subset $I\in \binom{[n]}{k} $ is the minimum $I=\min\{I_i^l\vert I\leq I_i^l\}$ for the subsets $I^l_i$ constructed above.
\end{lemma}

\begin{proof}
	We prove this claim by induction on the number of maximal intervals in $I$. For $I\in\binom{[n]}{k}$ let $I=I_1\cup \dots\cup I_m$ be the decomposition into $m$ maximal intervals. Consider the two subsets $J_1\defeq \{\min{I_2}-\sharp I_1 \dots,\min{I_2}-1 \}\cup I_2\cup\dots\cup I_m$  and $J_2\defeq I_1\cup \{n-1-k+\sharp I_1,\dots, n-1\}$. Now we easily see that $\min\{J_1,J_2\}=I$ where $J_1$ has $m-1$ maximal intervals and $J_2=I^l_i$ for $l=\sharp I_1$ and $i= l + \min I_1$.
\end{proof}

\begin{lemma}
\label{lem_im_of_lat_condition}
	Fix a natural number $N$, the standard lattice $\Lambda_0=\langle e_i\rangle_{i\in [N]}$ of $\quot^N$ and a second lattice $\Lambda=\langle \unif^{m_i} e_i\rangle_{i\in [N]}$ with $ m_i\in \Z$. Take lattices $\left\{\Lambda_j\right\}_{j\in J}$ representing the classes in $\overline{\{[\Lambda_0],[\Lambda]\}}$ and which are the unique maximal representatives in their homothety class contained in $\Lambda_0$. Then the set $\left\{m_i\vert i\in [N] \right\}$ is an interval if and only if the images $\im{\Lambda_{j,\res}\rightarrow \Lambda_{0,\res}}$ for $j\in J$ are pairwise different. In this case the class $[\Lambda]$ is fully determined by the set $\left\{\im{\Lambda_{j,\res}\rightarrow \Lambda_{0,\res}}\right\}_{j\in J}$ of images. 
\end{lemma}

\begin{proof}
	Fix a lattice $\Lambda=\langle \unif^{m_i} e_i\vert m_i\in \Z\rangle_{i\in [N]}$ maximal in its homothety class with $\Lambda\subseteq\Lambda_0$. We have $m_i\geq 0$ for all $i\in[N]$ since $\Lambda\subseteq\Lambda_0$ and $\min\{m_i\vert i\in[N]\}=0$ since $\Lambda$ is chosen to be maximal. The suitable representatives of the classes of the convex closure of $\{[\Lambda_0],[\Lambda]\}$ are of the form
	\begin{align*}
		\unif^{-l}\Lambda \cap \Lambda_{0} = \langle \unif^{m(l,i)}e_i \rangle_{i\in [N] } \text{ with }  m(l,i)=\max\{m_i-l,0\} \text{ for all } i\in [N]
	\end{align*}
	for $l\in [ \max\{m_i\vert i\in [n]\}]$. The corresponding images are
	\begin{align*}
		\im{\unif^{-l}\Lambda \cap \Lambda_{0}\rightarrow \Lambda_0}_{\res}= \langle e_i\vert m(l,i)=0\rangle= \langle e_i\vert m_i\leq l \rangle
	\end{align*}
	 and hence are pairwise different if and only if $\{m_i\vert i\in[N]\}$ is an interval.\\
	 In this case $\Lambda$ is uniquely determined since we get 
	 \begin{align*}
	 	m_i=\min\{l\vert e_i\in \im{\unif^-l\Lambda \cap \Lambda_{0}\rightarrow \Lambda_0}_{\res}\}
	 \end{align*} 
	  for all $i\in [N]$.
\end{proof}

\begin{lemma}
\label{lem_im_of_lat_condition_stand}
	For $\Lambda=\langle \unif^{-m_I}e_{I}\rangle\in \overline{\bigwedge^k \Gamma^{\mathrm{st}}}$ and $I,J\in\binom{[n]}{k}$ such that $J$ is maximal with $J<I$, we get $m_I-m_J\in \{0,1\}$. In particular the set $\{m_I\vert I\in\binom{[n]}{k}\}$ is an interval.
\end{lemma}

\begin{proof}
	First we recall that a lattice $\bigwedge^k \Lambda_i$ with $\Lambda_i\in \Gamma^{st}$ has the form $\bigwedge^k \Lambda_i=\langle \unif^{-m^{i}_I}e_I \rangle_{I\in \binom{[n]}{k} }$ with $ m^{i}_I= \sharp(I\cap [i])$ and for $J\leq I$ in $\binom{[n]}{k}$ we get $m^i_I-m^i_J= \sharp(I\cap [i]) - \sharp(J\cap [i])$. Now if $J$ is maximal with $J< I$ the two sets differ by just one element and hence $\sharp(I\cap [i]) - \sharp(J\cap [i])$ takes values in $\{0,1\}$.\\
	If for two lattices $\Lambda_1 ,\Lambda_2 \in \overline{\bigwedge^k\Gamma^{st}}$ with $\Lambda_i=\langle \unif^{m_I(\Lambda_i)}e_I\rangle$ for $i\in \{1,2\}$ and such that for all $I,J\in \binom{[n]}{k}$ and $J$ maximal with $J< I$ we have $m_I(\Lambda_i)- m_J(\Lambda_i)\in\{0,1\}$, then 
	\begin{align*}
		m(\Lambda_1\cap\Lambda_2)_{I}-m(\Lambda_1\cap\Lambda_2)_J&= \min\{m(\Lambda_i)_I\vert i=1,2 \}- \min\{m(\Lambda_i)_J\vert i=1,2 \}\\ & =\min\{m(\Lambda_i)_I - m(\Lambda_i)_J \vert i=1,2 \} \in \{0,1\}
	\end{align*} 
	and hence the lemma is proven for all $[\Lambda]\in \overline{\bigwedge^k \Gamma^{\mathrm{st}}}$.
\end{proof}

\begin{rem}
\label{rem_standart_shift}
	Using Lemma \ref{lem_im_of_lat_condition_stand} and Lemma \ref{lem_im_of_lat_condition} every lattice $[\Lambda]\in \overline{\bigwedge^k\Gamma^{st}}$ is determined by the images $\im{(\unif^l\Lambda \cap \bigwedge^k \Lambda_{0})_{\res}\rightarrow \bigwedge^k\Lambda_{0,\res}}$ for varying $l$.\\
	But in Lemma \ref{lem_images_are_linear_subsp} we already computed the images $\im{(\unif^l\bigwedge^k \Lambda_{i}\cap \bigwedge^k \Lambda_{0})_{\res}\rightarrow \bigwedge^k \Lambda_{0,\res}} = V_{I^l_i}$ for $i\in [n]$ and $0\leq l\leq k,i$. In particular for a fixed $i\in  [n]$ these images for all $0\leq l\leq \min\{k,i\}$ are determined by the image for $l=\min\{k,i\}$. In the lemma below we will generalise this to arbitrary lattices in $\overline{\bigwedge^k\Gamma^{st}}$.
\end{rem}

\begin{example}
	For $n=5$ and $k=2$ we will illustrate the last remark by the example $\unif^2\bigwedge^2\Lambda_2$. Using the basis $\{e_I\vert I\in \binom{[5]}{2} \}$ this lattice is spanned by $e_{\{0,1\}}$, $\unif e_{\{i,j\}}$ for $0\leq i\leq 1< j\leq 4$ and $\unif^2 e_{\{i,j\}}$ for $2\leq i< j\leq 4$.\\
	The image $\im{\unif^2\bigwedge^2 \Lambda_{2,\res}\rightarrow \bigwedge^2 \Lambda_{0,\res}}$ is the submodule spanned by $e_{\{0,1\}}$ and we recover\linebreak $I_{2}^2=\{0,1\}$. Now the image $\im{(\unif\bigwedge^2 \Lambda_{2}\cap\bigwedge^2\Lambda_0 )_{\res}\rightarrow \bigwedge^2 \Lambda_{0,\res}}$ is generated by the $e_{\{i,j\}}$ with $i\leq 1$. But this is equivalent to $\{i,j\}\leq \{1,4\}=I_2^1$. Finally $\unif^2\bigwedge^2 \Lambda_{2}$ is contained in $\bigwedge^2 \Lambda_{0}$ and hence the image $\im{(\bigwedge^2 \Lambda_{2}\cap\bigwedge^2\Lambda_0 )_{\res}\rightarrow \bigwedge^2 \Lambda_{0,\res}}$ is spanned by all $e_I$ and we recover $I_2^0=\{3,4\}$.
\end{example}

\begin{lemma}
\label{lem_pl_images_and_uniquenes}
\label{lem_image_determines_lattice}
Fix two classes $[\Lambda_1]\neq [\Lambda_2]$ in $\overline{\bigwedge^k\Gamma^{st}}$ and representatives $\Lambda_i$ for $i=1,2$ maximal in their class with $\Lambda_i\subseteq\bigwedge^k\Lambda_0$. Then the images $\im{\Lambda_{1,\res}\rightarrow \bigwedge^k \Lambda_{0,\res}}$ and $\im{\Lambda_{2,\res}\rightarrow \bigwedge^k \Lambda_{0,\res}}$ are different.	
\end{lemma}

\begin{proof}	
	For a subset $I\in \binom{[n]}{k}$ let us define a shift $I[l]$ for $l\in \N$ as follows. Using Lemma \ref{lem_every_subset_is_a_minimum} every subset $I$ is of the form $I=\min\{I_i^l\vert I\leq I_i^l\}$. Now the shift is $I[1]\defeq \min\{I_{i}^{l-1}\vert I\leq I_i^l\}$. Note that for every $i_0\in [n]$ and  $0<l_0\leq \min\{k,i\}$
	\begin{align*}
		I^{l_0}_{i_0}[1]=\min\{I_i^{l-1}\vert I^{l_0}_{i_0}\leq I_i^l  \}= \min\{I_i^{l}\vert I^{l_0-1}_{i_0}\leq I_i^l  \} =I^{l_0-1}_{i_0} .
	\end{align*}
	 For $I$ in $\binom{[n]}{k}$ we use that for $i$ and $0< l\leq \min\{i,k\}$ we have 
	 \begin{align*}
	 	I=\{i_0,\dots,i_{n-1}\}\leq I_i^l=\left\{i-l,\dots, i-1 \right\}\cup \left\{ n-1-k+l ,\dots,n-1\right\}
	 \end{align*}  
	 if and only if $i_{l-1}\leq i-1$. In particular for two subsets $I=\{i_0,\dots,i_{n-1}\}$ and $J=\{j_0,\dots,j_{n-1}\}$ in $\binom{[n]}{k}$ we see that for $\min\{I,J\}= \{\min\{i_0,j_0\},\dots,\min\{i_{n-1},j_{n-1}\}\}$ we get 
	 \begin{align*}
	 	\min\{I,J\}\leq I_i^l \Leftrightarrow \min\{i_{l-1},j_{l-1}\}\leq i-1\Leftrightarrow i_{l-1}\leq i-1 \text{ or } j_{l-1}\leq i-1  \Leftrightarrow I\leq I_i^l \text{ or } J\leq I_i^l.
	 \end{align*}
	 Hence we get an equality of sets
	  \begin{align*}
	  	\{I_i^l\vert \min\{I,J\}\leq I_i^l\}=\{I_i^l\vert I\leq I_i^l \text{, or }J\leq I_i^l \}
	  \end{align*} 
	   and derive 
	 \begin{align*}
	 	\min\{I,J\}[1] = \min\{I_i^{l-1}\vert \min\{I,J\}\leq I_i^l \} = \min\{I_i^{l-1}\vert I\leq I_i^l \text{, or }J\leq I_i^l  \} =\min\{I[1],J[1]\}.
	 \end{align*}  
	 Inductively we now define $I[l]$ to be the shift $I[l-1][1]$.\\
	All representatives of a class $[\Lambda]$ in  $\overline{\bigwedge^k\Gamma^{st}}$ have the form $\Lambda=\bigcap_{j\in J} \unif^{l_j} \bigwedge^k\Lambda_{j}$ for some $J\subseteq[n]$ and $l_j\in \Z$. If we now take the unique representative maximal with $\Lambda\subseteq\bigwedge^k\Lambda_0$ we can can assume that $0\in J$, $l_0=0$ and $l_i>0$ for $i\in J\setminus\{0\}$. If we note that 
	\begin{align*}
		 \unif^{-1}\Lambda\cap\bigwedge^k\Lambda_0 = \bigcap\limits_{0\neq j\in J} \unif^{l_j-1} \bigwedge^k\Lambda_{j}\cap\bigwedge^k\Lambda_0
	\end{align*} 
	and $V_{I_0^0}$ is $\bigwedge^k\Lambda_0$ we now get
	\begin{align*}
		&\im{(\unif^{-1}\Lambda\cap \bigwedge^k \Lambda_{0})_{\res}\rightarrow \bigwedge^k \Lambda_{0,\res}}= \bigcap\limits_{0\neq j\in J} \im{(\unif^{{-1}+l_j}\bigwedge^k \Lambda_{j}\cap \bigwedge^k \Lambda_{0})_{\res}\rightarrow \bigwedge^k \Lambda_{0,\res}}\\
		&= \bigcap\limits_{0\neq j\in J} V_{I_{j}^{{-1}+l_j}}=\bigcap\limits_{0\neq j\in J} V_{I_{j}^{{-1}+l_j+1}[1]} =V_{I(\Lambda)[1]}.
	\end{align*}
	Inductively we can determine the images for all lattices in $\overline{\{\Lambda,\bigwedge^k\Lambda_0\}}$ from the image of $\Lambda$ and hence the class $[\Lambda]$ is fully determined by its image.
\end{proof}

\begin{example}
	Let us give an example for the last lemma. We fix $n=6$ and $k=3$ and want to find the unique lattice $\Lambda$ in $\bigwedge^3 \Gamma^{\mathrm{st}}$ with $I({\Lambda})=\{0,2,3\}$. We write $\Lambda=\bigcap_{i\in [6]}\unif^{n_i}\bigwedge^3\Lambda_i$. If $\Lambda$ is maximal in its homothety class with $\Lambda\subseteq\bigwedge^3\Lambda_0$ then it is easy to see that $n_0=0$ and $n_i\geq 0$ for $i\neq 0$. Write $\Lambda$ in the basis of $\bigwedge^3\quot^6$ as $\langle \unif^{m_I}e_I\vert I\in\binom{[6]}{3}\rangle$. By assumption we have $m_I=\max\left\{n_i-\sharp I\cap[i]\mid i\in [6] \right\}=0$ exactly when $I\leq I(\Lambda)$ hence $m_I=0$ precisely if $I\leq I_i^{n_i}$ for all $i$. Using the notation from Remark \ref{rem_standart_shift} we get $I(\Lambda)=\min_i\{I_i^{n_i}\}$. Note that the subsets smaller than $I(\Lambda)$ are $\{0,1,2\}$, $\{0,1,3\}$ and $\{0,2,3\}$.\\
	Now the lattice $\unif^{-1}\Lambda\cap \bigwedge^2 \Lambda_{0}$ is of the form $\langle \unif^{m'_I}e_I\vert I\in\binom{[6]}{3}\rangle$ with $m'_I=\max\{0,m_I-1\}$. Hence the image of $\im{(\unif^{-1}\Lambda\cap \bigwedge^2 \Lambda_{0})_{\res}}$ is generated by the $e_I$ with $n_i-\sharp I\cap[i]\leq 1$ for all $i$. This is equivalent to $I\leq\min\{I_i^{n_i-1}\vert i\in [6]\}=I(\Lambda) [1]$. But on the other hand the subsets of the form $I_i^{l-1}$ with $I(\Lambda) \leq I_i^{l}$ are the following: $\{2,3,5\}$, $\{2,4,5\}$ and $\{3,4,5\}$. Hence the shift is simply $I(\Lambda)[1]=\{2,3,5\}$. Shifting again we get $I(\Lambda)[2]=\{3,4,5\}$ and hence $m_I\leq 2$ for all $I$.\\
	 We can now use the explicit values $m_I$ for $I\in\binom{[6]}{3}$ and calculate the exponents $n_i$ for $i\in [6]$:
	\begin{enumerate}
		\item for $i=1$ and $I\leq I(\Lambda)$: $0= m_I\geq n_1-\sharp I\cap [1]=n_1-1$ hence $n_1=1$
		\item for $i=2$ and $I=\{0,2,3\}$: $0= m_I\geq n_2-\sharp I\cap [2]=n_2-1$ hence $n_2=1$
		\item for $i=3$ and $I=\{0,2,3\}$: $0=m_I\geq n_3-\sharp I\cap [3]=n_3-2$ hence $n_3\leq 2$ 
		\item for $i=4$ and $I=\{2,3,5\}$: $1= m_I\geq n_4-\sharp I\cap [4]=n_4-2$ hence $n_4\leq 3$ 
		\item for $i=5$ and $I=\{0,2,3\}$: $0=m_I\geq n_5-\sharp I\cap [5]=n_3-3$ hence $n_3\leq 3$ 
		\item for $I=\{2,4,5\}$ and $i\neq 4$: $n_i-\sharp I\cap [i]\leq 1<m_I$ hence $2=n_4-\sharp I \cap [4]=n_4-1$ and $n_4=3$
	\end{enumerate}
	But now we note that $\unif\bigwedge^2\Lambda_1\subseteq \unif\bigwedge^2\Lambda_2\cap \unif\bigwedge^2\Lambda_3\cap \bigwedge^2\Lambda_0$ and $\unif^3 \bigwedge^2\Lambda_4\subseteq \unif^2\bigwedge^2\Lambda_3\cap \unif^3\bigwedge^2\Lambda_5$. Hence we get $\Lambda=\unif\bigwedge^2\Lambda_1\cap \unif^3 \bigwedge^2\Lambda_4$.\\
	Now we can check $I_{\bigwedge^2\Lambda_1}=I_2^1=\{0,4,5\}$ and $I_{\unif^3 \bigwedge^2\Lambda_4}=I_4^3=\{1,2,3\}$ and hence $I(\Lambda)$ is\linebreak $\min\{\{0,4,5\},\{1,2,3\}\}=\{0,2,3\}$.
\end{example}

\begin{definition}
	For $I$ in $\binom{[n]}{k}$ Lemma \ref{lem_image_determines_lattice} gives us a unique lattice in $\overline{\Gamma^{\mathrm{st}}}$ such that the image $\im{\Lambda_{I,\res}\rightarrow \bigwedge^k \Lambda_{0,\res}}$ is $I$. In the following this lattice will be denoted by $\Lambda_I$. Using Lemma \ref{lem_irr_comp_and_im_of_lat} we now get a unique irreducible component of $\M{\overline{\Gamma^{\mathrm{st}}}}_{\res}$ surjecting to $\Projsp(\Lambda_I )$. This component will be denoted by $C_I$.
\end{definition}

\begin{prop}
\label{prop_im_of_irr_comp}
	We get a bijection
	\begin{align*}
		\left\{C\middle\vert C \text{ irr. component of } \M{\overline{\bigwedge^k\Gamma^{\mathrm{st}}}}_{\res} \right\}&\longrightarrow \left\{\Projsp (V_I)\subseteq \Projsp(\bigwedge^k \Lambda_0)_{\res}\middle\vert I \in \binom{[n]}{k} \right\}\\
		C & \longmapsto \projectbar{0}(C).
	\end{align*}
	And moreover for a linear subspace $\Projsp(V_I)$ for $I\in \binom{[n]}{k}$ the inverse image $(\projectbar{0})^{-1}(\Projsp(V_I) )$ is the union $\bigcup_{J\leq I} C_J$ of irreducible components. 
\end{prop}

\begin{proof}
	In Lemma \ref{lem_convex_irr_com_correspond_to_lattices} we showed that the number of irreducible components of $\M{\Gamma}_{\res}$ coincides with the number of elements in $\Gamma$ whenever $\Gamma$ is a convex set of lattice classes. Using Lemma \ref{lem_image_determines_lattice} we now get $\binom{n}{k}$ as the number of irreducible components of $\M{\overline{\Gamma^{\mathrm{st}}}}$.\\
	Using Lemma \ref{lem_images_are_linear_subsp} the images $\projectbar{0}(C)$ for an irreducible component $C$ in $\M{\overline{\bigwedge^k\Gamma^{\mathrm{st}}}}_{\res}$ are of the form $\Projsp(V_I)$ for some $I\in\binom{[n]}{k}$. Now by Lemma \ref{lem_image_determines_lattice} the above map is injective and hence bijective by a cardinality argument. \\
	The second part of the statement now follows directly from Lemma \ref{lem_inverse_of_center} using that an irreducible component $C_J$ for some $J$ in $\binom{[n]}{k}$ is mapped into $\Projsp(V_I)$ if and only if $\projectbar{0}(C_J)=\Projsp(V_J)$ is contained in $\Projsp(V_I)$ which is equivalent to $J\leq I$.
\end{proof}

\pagebreak[1]
\subsection{Irreducible components and the convex closure}
\label{sec_conj}
	With the last proposition we got a quite precise relation between the irreducible components of $\M{\overline{\bigwedge^k\Gamma^{\mathrm{st}}}}_{\res}$ and linear subspaces in $\Projsp(\bigwedge^k\Lambda_0 )$. Since we have a similar relation for the irreducible components of $\Mgr{k}{\Gamma^{\mathrm{st}}}_{\res}$ and the Schubert varieties in $\gr{\Lambda_0}$, the number of irreducible components of $\Mgr{k}{\Gamma^{\mathrm{st}}}_{\res}$ and $\M{\overline{\bigwedge^k\Gamma^{\mathrm{st}}}}_{\res}$ are the same. To make this relation more precise we need to specify an explicit bijection between the sets of irreducible components of $\M{\overline{\bigwedge^k\Gamma^{\mathrm{st}}}}_{\res}$ and $\M{{\bigwedge^k\Gamma^{\mathrm{st}}}}_{\res}$. This is done in the conjecture below.

\begin{conj}
\label{conj_convexcl_bijection}
	We get a bijection
	\begin{align*}
		\left\{C\middle\vert C \text{ irr. component of } \M{\overline{\bigwedge^k\Gamma^{\mathrm{st}}}}_{\res} \right\}&\longrightarrow \left\{C\middle\vert C \text{ irr. component of } \M{\bigwedge^k\Gamma^{\mathrm{st}}}_{\res} \right\}\\
		C & \longmapsto \projbar (C) 
	\end{align*}
	where $\projbar\colon \M{\overline{\bigwedge^k\Gamma^{\mathrm{st}}}}\rightarrow \M{\bigwedge^k\Gamma^{\mathrm{st}}}$ is the natural projection.
\end{conj}

\begin{rem}
	We already know from Lemma \ref{lem_irr_comp_birational_map} that for an irreducible component $C$ of $\M{\bigwedge^k\Gamma^{\mathrm{st}}}_{\res}$ there exist a unique irreducible component $\overline{C}$ of $\M{\overline{\bigwedge^k\Gamma^{\mathrm{st}}}}_{\res}$ surjecting to $C$. Hence to prove the conjecture we just need to show that the images $\projbar (C)$ are irreducible components.
\end{rem}

As evidence for the conjecture we have the following two proposition. We will omit the proof of the first proposition and refer to \cite[Appendix C]{Gor19} for explicit calculations using Sage.
\begin{prop}
	Conjecture \ref{conj_convexcl_bijection} is true for $n\leq 7$.
\end{prop}

\begin{prop}
\label{prop_conj_k_2}
	For $k=2$ Conjecture \ref{conj_convexcl_bijection} is true.
\end{prop}

The proof of the second proposition will occupy us for the rest of this section therefore let us first indicate one important immediate consequence of the Conjecture. Furthermore we will describe a method to approach the conjecture in general before we prove the proposition.

\begin{thm}
   \label{thm_components_linear_subsp}
	Assume Conjecture \ref{conj_convexcl_bijection}. Then we get a bijection
	\begin{align*}
		\left\{C\middle\vert C \text{ irr. component of } \M{\bigwedge^k\Gamma^{\mathrm{st}}}_{\res} \right\}&\longrightarrow \left\{\Projsp (V_I)\subseteq \Projsp(\bigwedge^k \Lambda_0)_{\res}\middle\vert I \in \binom{[n]}{k} \right\}\\
		C & \longmapsto \project{0} (C)
	\end{align*}
	where $\project{0}\colon \M{\bigwedge^k\Gamma^{\mathrm{st}}}\rightarrow \Projsp(\bigwedge^k \Lambda_0)$ is the natural projection. 
	And moreover for a linear subspace $\Projsp(V_I)$ for $I\in \binom{[n]}{k}$ the inverse image $(\project{0})^{-1}(\Projsp(V_I) )$ is the union $\bigcup_{J\leq I} C_J$ of irreducible components. 
\end{thm}

\begin{proof}
	Fix $I$ in $\binom{[n]}{k}$ and consider the linear subspace $\Projsp(V_I)$ of $\Projsp(\bigwedge^k \Lambda_0 )$. By Proposition \ref{prop_im_of_irr_comp} there exist a unique irreducible component $\overline{C}$ of $\M{\overline{\bigwedge^k\Gamma^{\mathrm{st}}}}_{\res}$ surjecting to $\Projsp(V_I)$. Now by\linebreak Conjecture \ref{conj_convexcl_bijection} the image $C\defeq \projbar (\overline{C})$ of $\overline{C}$ in $\M{\bigwedge^k\Gamma^{\mathrm{st}}}$ is an irreducible component surjecting to $\Projsp(V_I)$.\\
	Conversely for an irreducible component $C$ of $\M{\bigwedge^k\Gamma^{\mathrm{st}}}_{\res}$ there is by Lemma \ref{lem_irr_comp_birational_map} an irreducible component $\overline{C}$ of $\M{\overline{\bigwedge^k\Gamma^{\mathrm{st}}}}_{\res}$ surjecting to $C$. Hence by Proposition \ref{prop_im_of_irr_comp} the image of $C$ is of the form $\Projsp(V_I)$ for some $I$ in $\binom{[n]}{k}$.\\
	The second part of the statement follows similarly.
\end{proof}

\begin{rem}
	In \cite{AL17} a combinatorial method was described to compute dimensions of certain images of rational maps using a result by \cite{Li18}. To describe this method and the implication we want to use, we need the following setup. Note that we use the dual notion of projective space compared to the reference.\\
	Fix a finite set of lattice classes $\Gamma$ in $\quot^n$, an irreducible component $C$ of $\M{\Gamma}_{\res}$ and a class $[\Lambda]$ in $\Gamma$. Take a representative $\Lambda=\langle \unif^{m_I (\Lambda )} e_I \rangle$ of $[\Lambda]$ and a representative $\Lambda_C= \langle \unif^{m_I(\Lambda_C )}e_I\rangle$ of the class corresponding to the irreducible component $C$. Choose $\Lambda_C$ to be maximal with $\Lambda_C \subseteq\Lambda$. We define the subset 
	\begin{align*}
		W_{\Lambda}\defeq\{i\in[n]\vert m(\Lambda)_i-m(\Lambda_C )_i< \max\limits_{j\in [n]}\{m(\Lambda)_j-m(\Lambda_C )_j\} \}
	\end{align*}  of $[n]$ and construct the set 
	\begin{align*}
		M(h,C)\defeq\{(a_\Lambda)_{[\Lambda]\in \Gamma}\in\N^{\Gamma}\vert \sum_{\Gamma}a_\Lambda =h\text{ and } n-\sum_{\Lambda\in  I}a_\Lambda > \sharp\bigcap_{\Lambda\in I}W_{\Lambda}\text{ for all subsets } \emptyset\neq I \subseteq\Gamma \}.
	\end{align*} 
	The following result describes how this combinatorial data encodes the dimension of the image $\projbar (C)$.
\end{rem}

\begin{prop}(\cite{Li18} and  \cite[Theorem 2.18]{AL17})
\label{prop_dimension_computed_with_subsets}
	For a finite set of lattice classes $\Gamma$ and an irreducible component $C$ of $\M{\overline{\Gamma}}_{\res}$ the dimension of $\projbar (C)\subseteq \M{{\Gamma}}_{\res}$ is computed by
	\begin{align*}
		\dim(\projbar (C))= \max\{h\vert M(h,C)\neq \emptyset \}.
	\end{align*} 
\end{prop}

\begin{rem}
\label{rem_idea_proof_conjecture}
	Let us explain an approach to Conjecture \ref{conj_convexcl_bijection} using the proposition above. Fix an irreducible component $C$ of $\M{\overline{\bigwedge^k\Gamma^{\mathrm{st}}}}$ and take the corresponding class $[\Lambda_C]$ in $\overline{\bigwedge^k\Gamma^{\mathrm{st}}}$. Since the image $\projbar (C)$ is clearly irreducible, it is enough to show $\dim(\projbar (C))=\binom{n}{k}-1$. It is now possible to apply Proposition \ref{prop_dimension_computed_with_subsets} directly for $\Gamma= \bigwedge^k\Gamma^{\mathrm{st}}$ but we also might first do the following reduction steps and apply the proposition to a simpler set of lattice classes.\\ 
	 We recall that using Lemma \ref{lem_irr_comp_birational_map} the image of $C$ in $\M{\overline{\bigwedge^k\Gamma^{\mathrm{st}}}}_{\res}$ is clearly an irreducible component if $[\Lambda_C]$ is already in $\bigwedge^k\Gamma$. Hence we just have to check the cases where $[\Lambda_C]$ is in $\overline{\bigwedge^k\Gamma^{\mathrm{st}}}\setminus {\bigwedge^k\Gamma^{\mathrm{st}}}$. \\
	Now we can find a subset $\Gamma_C$ of $\bigwedge^k\Gamma^{\mathrm{st}}$ minimal such that $\overline{\Gamma_C}$ contains $[\Lambda_C]$. With out loss of generality we can further assume that the homothety class of $\bigwedge^k\Lambda_0$ is contained in $\Gamma_C$. By Lemma \ref{lem_irr_comp_birational_map} and Lemma \ref{lem_convex_irr_com_correspond_to_lattices} the image of $C$ in $\M{\overline{\Gamma_C}}_{\res}$ is an irreducible component. And using Lemma \ref{lem_irr_comp_birational_map} again we see that the image of $C$ in $\M{\Gamma_C}_{\res}$ is an irreducible component if and only if the image in $\M{\bigwedge^k\Gamma^{\mathrm{st}}}_{\res}$ is an irreducible component. \\
	We now reduced to compute the dimension of the image of $C$ in $\M{\Gamma_C}$. Applying Proposition \ref{prop_dimension_computed_with_subsets} to $\Gamma=\Gamma_C$, this is equivalent to $M(\binom{[n]}{k}-1,C)$ not being empty.\\
	In general the sets $\Gamma_C$ can be difficult to determine, but for $k=2$ we have the following easy description. 
	\end{rem}

\begin{lemma}
\label{lem_convex_closure_k_2}
	For $[\Lambda]\in \overline{\bigwedge^2 \Gamma^{\mathrm{st}}}$ there are classes $[\Lambda']$ and $[\Lambda'']$ in $\bigwedge^2 \Gamma^{\mathrm{st}}$ such that $[\Lambda]$ is in the convex closure $\overline{\left\{[\Lambda'],[\Lambda'']\right\}}$.
\end{lemma}

\begin{proof}
	Take $[\Lambda]$ in $\overline{\bigwedge^2 \Gamma^{\mathrm{st}}}$ arbitrary. Then $\Lambda$ is of the form $\bigcap_{i\in I} \unif^{n_i}\bigwedge^2\Lambda_i$ for some $I\subseteq[n]$ and without loss of generality we have $0\in I$, $n_0=0$ and $n_i> 0$ for all $i\in I\setminus\{0\}$. Further assume that $I$ is minimal, i.e. $\Lambda$ is properly contained in $\bigcap_{i\in J} \unif^{n_i}\bigwedge^2\Lambda_i$ for all proper subsets $J\subset I$. Now we note that for all $i\in I$ we have $\unif^{n_i}\bigwedge^2\Lambda_i\subseteq \bigwedge^2\Lambda_0$ if $n_i\geq 2$. Using the minimality of $I$ we conclude that $n_i=1$ for all $i\in I\setminus \{0\}$. But since $p \bigwedge^2\Lambda_i\subseteq \unif\bigwedge^2\Lambda_j$ for $i\leq j$ we again see by the minimality of $I$ that $\sharp I\leq 2$.
\end{proof}

We can now prove Conjecture \ref{conj_convexcl_bijection} for $k=2$.

\begin{proof}[Proof of Proposition \ref{prop_conj_k_2} ]
	Fix an irreducible component $C$ of $\M{\bigwedge^2\Gamma^{\mathrm{st}}}$. Following the idea described in Remark \ref{rem_idea_proof_conjecture} and Lemma \ref{lem_convex_closure_k_2} we just have to prove $M(\binom{n}{2}-1,C)\neq \emptyset$ for\linebreak $\Gamma=\left\{[\bigwedge^2 \Lambda_0],[\bigwedge^2\Lambda_i]\right\}$ for some $i\in [n]$.\\
	But for $\unif\bigwedge^2 \Lambda_0\cap\bigwedge^2\Lambda_i=\langle \unif^{m^{i}_I}e_I\vert m^{i}_I= \max\{\sharp(I\cap [i]),1\} \rangle_{I\in \binom{[n]}{k} }$ hence $W_{\bigwedge^2\Lambda_0}= \left\{I\vert m^i_{I}>1 \right\}$ and $W_{\bigwedge^2\Lambda_i}= \left\{I\vert m^i_{I}<1 \right\}$. Now for $a_{\bigwedge^2 \Lambda_0}\defeq \sharp W_{\bigwedge^2 \Lambda_i}$ and $a_{\bigwedge^2 \Lambda_i}\defeq \sharp W_{\bigwedge^2 \Lambda_0}+ \sharp\left\{I\vert m^i_{I}=1 \right\}-1$ we have an element $(a_{\bigwedge^2 \Lambda_0},a_{\bigwedge^2 \Lambda_i})\in M(\binom{n}{2}-1,C)$.

\end{proof}

\subsection{The Pl\"ucker embedding}

To relate the local model as a Mustafin variety to the well understood Mustafin varieties for projective spaces, we want to use the Pl\"ucker embedding. For every lattice $\Lambda_i$ in $\Gamma^{\mathrm{st}}$ we consider the Pl\"ucker embedding 
\begin{align*}
	\mathrm{Pl}_{\Lambda_i} \colon \gr{\Lambda_i }\longrightarrow \Projsp(\bigwedge^k \Lambda_i )
\end{align*}
and the projections $\projectgr{i}\colon \Mgr{k}{\Gamma^{\mathrm{st}}}\rightarrow \gr{\Lambda_i}$ and $\project{i}\colon \M{\bigwedge^k\Gamma^{\mathrm{st}}}\rightarrow \Projsp(\bigwedge^k \Lambda_i )$.

Together we now get a commutative diagram
\begin{align*}
	\xymatrix{
	 \Mgr{k}{\Gamma^{\mathrm{st}}} \ar[r]^{\mathrm{Pl}_{\Gamma^{\mathrm{st}}}} \ar@{_(->}[d]_{\prod \projectgr{i}} & \M{\bigwedge^k\Gamma^{\mathrm{st}}} \ar@{_(->}[d]^{\prod \project{i}}\\
	 \prod\limits_{[\Lambda_i]\in \Gamma^{\mathrm{st}} } \gr{\Lambda_i} \ar[r]^{\prod \mathrm{Pl}_{\Lambda_i} } & \prod\limits_{[\bigwedge^k \Lambda_i]\in \bigwedge^k \Gamma^{\mathrm{st}} } \Projsp(\bigwedge^k \Lambda_i )
	}
\end{align*}
where we get the map $\mathrm{Pl}_{\Gamma^{\mathrm{st}}}$ since we have it on the generic fiber and $\Mgr{k}{\Gamma^{\mathrm{st}}}$ is flat.
\begin{rem}
	 In \cite[Proposition 4.5]{Ha11} (or \cite[discussion after Definition 2.1]{Hae14}) it was claimed, that the diagram above constructed for every convex set of lattice classes $\Gamma$ is cartesian, i.e. inside $\prod_{[\Lambda]\in \Gamma}  \Projsp(\bigwedge^k \Lambda )$ the Mustafin variety $\Mgr{k}{\Gamma}$ is the intersection of $\prod_{[\Lambda]\in \Gamma} \gr{\Lambda}$ and $\M{\bigwedge^k\Gamma}$. \\
	 Let us illustrate in the example $\{[\Lambda_0],[\Lambda_1]\}\subseteq\Gamma^{\mathrm{st}}$ that this claim is not true.
	  To give an explicit example for $n=4$ and $k=2$ let us identify $\Lambda_{i}$ for $i\in \{0,1\}$ with $\rg^4$ with basis $\{e_i\}_{i=0,\dots ,  3}$. The map $\Lambda_{0}\rightarrow\Lambda_{1}$ is now given by multiplying the basis vector $e_0$ with $\unif$. Further we use the order $e_0\wedge e_1, e_0\wedge e_2,e_0\wedge e_3, e_1\wedge e_2, e_1\wedge e_3, e_2\wedge e_3 $ for the basis of $\bigwedge^2 \Lambda_{i}= \bigwedge^2 \rg^4$. \\ 
	  First we recall the moduli descriptions of the two Mustafin varieties. For an $\rg$-algebra $R$ the $R$-valued points of $\Mgr{2}{\{[\Lambda_0],[\Lambda_1]\}}$ are given by 
	  \begin{align*}
	  	\{(\mathcal{F}_0,\mathcal{F}_1)\in \mathrm{Gr}_{2,4}(R) ^2:  \mathrm{diag}(\unif,1,1,1)\mathcal{F}_0 \subseteq \mathcal{F}_1, \mathrm{diag}(1,\unif,\unif,\unif)\mathcal{F}_1 \subseteq \mathcal{F}_0 \}
	  \end{align*}
	  where the flatness of the moduli functor was proven in \cite{G01}. Since $\left\{\left[\bigwedge^2 \Lambda_0\right],\left[\bigwedge^2 \Lambda_1\right]\right\}$ is convex, also $\M{\left\{\left[\bigwedge^2 \Lambda_0\right],\left[\bigwedge^2 \Lambda_1\right]\right\}}$ has a moduli description. It is cut out of $\Projsp^5(R)^2$ by conditions imposed by the two inclusions $\bigwedge^2 \Lambda_0\subseteq \bigwedge^2 \Lambda_1$ and $\bigwedge^2 \Lambda_1\subseteq \unif^{-1}\bigwedge^2 \Lambda_0$ (see \cite[Definition 4]{FA01}). The $R$-valued points of $\M{\{[\bigwedge^2 \Lambda_0],[\bigwedge^2\Lambda_1]\}}$ are given by
	  \begin{align*}
	  	 \{(\mathcal{L}_0,\mathcal{L}_1)\in \Projsp^5(R)^2:  \mathrm{diag}(\unif,\unif,\unif,1,1,1)\mathcal{L}_0 \subseteq \mathcal{L}_1, \mathrm{diag}(1,1,1,\unif,\unif,\unif)\mathcal{L}_1 \subseteq \mathcal{L}_0 \}.
	  \end{align*}  
	  Consider the pair $(\mathcal{F}_0, \mathcal{F}_1)$ in $ \mathrm{Gr}_{2,4}(\res) ^2$ with $\mathcal{F}_0=\langle e_0,e_{3} \rangle$ and $\mathcal{F}_1=\langle e_1,e_{2} \rangle$. This pair is not point in $\Mgr{2}{\{[\Lambda_0],[\Lambda_1]\}}(\res)$. But the pair $\left(\bigwedge^2 \mathcal{F}_0,\bigwedge^2 \mathcal{F}_1\right)= \left((0:0:1:0:0:0),(0:0:0:1:0:0)\right)$ is in $\M{\{[\bigwedge^2\Lambda_0],[\bigwedge^2\Lambda_1]\}}(\res)$. Hence we have an element in 
	  \begin{align*}
	  	\M{\{[\bigwedge^2\Lambda_0],[\bigwedge^2\Lambda_1]\}}(\res) \cap \prod_{i=0,1} \gr{\Lambda_i}(\res)
	  \end{align*}
	  which is not in $\Mgr{2}{\{[\Lambda_0],[\Lambda_1]\}}(\res)$. \\
	 But assuming Conjecture \ref{conj_convexcl_bijection} we can still get a relation in the following sense.
\end{rem}

\begin{prop}
\label{porp_pluecker_bijection_irr_comp}
	Assume Conjecture \ref{conj_convexcl_bijection}. Then the Pl\"ucker embedding 
	\begin{align*}
		\mathrm{Pl}_{\Gamma^{\mathrm{st}}}\colon\Mgr{k}{\Gamma^{st}}\longrightarrow \M{\bigwedge^k\Gamma^{st}}
	\end{align*}
	induces a bijection
	\begin{align*}
		\left\{C\middle\vert C \text{ irr. component of } \M{\bigwedge^k\Gamma^{\mathrm{st}}}_{\res} \right\} &\longrightarrow \left\{C\middle\vert C \text{ irr. component of } \Mgr{k}{\Gamma^{\mathrm{st}}}_{\res} \right\}\\
		C&\longmapsto C\cap \Mgr{k}{\Gamma^{\mathrm{st}}}_{\res}
	\end{align*}
	
	between the sets of irreducible components of the special fibres.
\end{prop}

Before we prove this proposition let us discuss some properties of the irreducible components of $\Mgr{k}{\Gamma^{st}}_{\res}$. The extended affine Weyl group $\widetilde{W}$ for $\gln$ is the semi direct product $W\times \Z^n$ of the finite Weyl group $W=S_n$ with $\Z^n$. The extended affine Weyl group is not a Coxeter group but we equipped with the Bruhat order induced by the Coxeter group $\affw$. Fix the minuscule coweight $\mu=(1^k,0^{n-k})$. A detailed discussion of these groups see \cite{KR00}. \\
	For a Coxeter group $(\mathcal{W},\mathcal{S})$ and some subset $S\subseteq\mathcal{S}$ denote by $\mathcal{W}_S$ the subgroup of $\mathcal{W}$ generated by $S$. For an element $x\in \mathcal{W}$ we denote by $x^S$ the unique element of minimal length in the orbit $\mathcal{W}_S x$. The set of these elements endowed with the quotient group structure is denoted by $\mathcal{W}^S$.\\
\begin{lemma}
\label{lem_schubertvar_determined_by_projections}
	An irreducible component $C$ of $\Mgr{k}{\Gamma^{st}}_{\res}$ is the intersection of $\prod \projectgr{i}(C) $ with $\Mgr{k}{\Gamma^{st}}_{\res}$ taken in $\prod_{i\in[n]}\gr{\Lambda_i}_{\res}$.
\end{lemma}

The idea for this lemma is to use a more general Lemma below which describes the Schubert varieties in the affine flag variety $\mathcal{F}$ using their projections $\pr_i$ to the affine Grassmannians $\mathcal{G}_i$ for $i\in[n]$. To apply this lemma we use the embedding $\Mgr{k}{\Gamma^{st}}_{\res}=\loc_{\res}\subseteq\mathcal{F}$ constructed in \cite[chapter 4.2]{G01} and the observation that under this embedding the irreducible components of $\Mgr{k}{\Gamma^{st}}_{\res}$ are Schubert varieties in $\mathcal{F}$ see \cite[Proposition 4.5 (iii)]{G01}.\\
Now fix a Schubert variety $\schuvar{w}$ in $\mathcal{F}$ for some $w\in \widetilde{W}$. This variety is the union $\bigcup_{w'\leq w}\schucell{w'}$ of Schubert cells. Since the inverse image $\pr_i^{-1}\left(\pr_i(\schuvar{w})\right)$ of the image of $\schuvar{w}$ in $\mathcal{G}_i$ is invariant under the Iwahori action for all $i$, it is a union of Schubert cells. In particular this union is precisely the union over all elements $w'$ in $\widetilde{W}$ congruent modulo $\affw_{{S_i}}$ to an element $w''\leq w$. With out loss of generality we can assume that $w$ is in $\affw$ and hence also $w''$ and $w'$ are in $\affw$. In particular for an Schubert variety $X_I$ of $\gr{\Lambda_0}_\res$ the inverse image $\left(\projectgr{0}\right)^{-1}(X_I)$ is the union $\bigcup_{I'\leq I} C_I$ of irreducible components.\\ 
 All the embeddings and projections above are compatible, i.e. we have the following commutative diagram:
	\begin{align*}
		\xymatrix{
		\Mgr{k}{\Gamma^{st}}_{\res} \ar@{^(->}[r]\ar[d]_{\projectgr{i}} & \mathcal{F}\ar[d]^{\pr_i}\\
		\gr{\Lambda_i}_{\cres}\ar@{^(->}[r]& \mathcal{G}_i
		}
	\end{align*}
	In particular it is enough to apply the following theorem and prove the lemma below on Schubert varieties of affine flag varieties.

\begin{thm}\cite[Lemma 3.6]{D77}
\label{thm_deo}
	Fix a Coxeter group $(\mathcal{W},\mathcal{S})$ and a family of subsets $\{S_i\}_i$ of $\mathcal{S}$ with $\bigcap_{i}S_i=\emptyset$. For $x,y$ in $\mathcal{W}$ we have $x\leq y$ if and only if $x^{S_i}\leq y^{S_i}$ for all $i$ in $[n]$.
\end{thm}

\begin{corollary}
\label{lem_aff_schu_intersection_of_projections}
	For $w\in \widetilde{W}$ the Schubert variety $\schuvar{w}$ in the affine flag variety $\mathcal{F}$ is the intersection $\left(\prod_{i\in[n]} \pr_i(\schuvar{w}) \right)\cap \mathcal{F}$ taken in the product $\prod_{i\in[n]}\mathcal{G}_i$.
\end{corollary}

We now prepared enough to prove Proposition \ref{porp_pluecker_bijection_irr_comp}.

\begin{proof}[Proof of Proposition \ref{porp_pluecker_bijection_irr_comp} ]
First recall from Lemma \ref{lem_sub_var_under_pl} that for $I\in\binom{[n]}{k}$ the Schubert variety $X_I$ in $\gr{\Lambda_0}$ is the intersection of the linear subspace $\Projsp(V_I)$ with the image of $\gr{\Lambda_0}$ under the Pl\"ucker embedding. Although this statement was formulated for Schubert varieties in $\gr{\Lambda_0}$ and the linear subspaces in $\Projsp(\bigwedge^k\Lambda_0 )$ it is certainly true for all lattices in $\Gamma^{\mathrm{st}}$ and we will make use of this fact.\\
Let us consider the following diagram
\begin{align*}
	\xymatrix{
	 	\Mgr{k}{\Gamma^{\mathrm{st}}}\ar[r]^{\mathrm{Pl}_{\Gamma^{\mathrm{st}}}} \ar[d]^{\projectgr{0}} 	& \M{\bigwedge^k\Gamma^{st}}\ar[d]^{\project{0}} \\
	 	\gr{\Lambda_0}\ar[r]^{\mathrm{Pl}_{\Lambda_0}} 		& \Projsp(\bigwedge^k\Lambda_0).	
	}
\end{align*}
	For an irreducible component $C_I^{\mathrm{gr}}$ of $\Mgr{k}{\Gamma^{st}}_{\res}$ the open part $C_I^{{\mathrm{gr}},\circ}\defeq C_I^{\mathrm{gr}}\setminus\bigcup\limits_{J< I} C_J^{\mathrm{gr}}$ is the inverse image $(\projectgr{0}) ^{-1}(X_I^{\circ})$ of the Schubert cell $X_I^{\circ}= X_I\setminus\bigcup_{J<I}X_J$. Furthermore  this is also the inverse image 
	\begin{align*}
		(\projectgr{0}\circ \mathrm{Pl}_{\Lambda_0})^{-1}(\Projsp(V_I)^{\circ})=(\mathrm{Pl}_{\Gamma^{\mathrm{st}}}\circ \project{0} )^{-1}(\Projsp(V_I)^{\circ})
	\end{align*}
	 of the open $\Projsp(V_I)^{\circ}= \Projsp(V_I)\setminus\bigcup_{J<I}\Projsp(V_J)$. Using Theorem \ref{thm_components_linear_subsp} we now identify $(\project{0})^{-1}(\Projsp(V_I)^{\circ})$ with the open part $C_I^{{\mathrm{pr}},\circ}\defeq C_I^{\mathrm{pr}}\setminus\bigcup\limits_{J< I} C_J^{\mathrm{pr}}$. In conclusion we get $C_I^{{\mathrm{gr}},\circ}= C_I^{{\mathrm{pr}},\circ}\cap\Mgr{k}{\Gamma^{\mathrm{st}}}$ and in particular $C_{I}^{\mathrm{gr}}\subseteq C_I^{{\mathrm{pr}}}\cap\Mgr{k}{\Gamma^{\mathrm{st}}}_{\res}$.\\
	On the other hand first using Lemma \ref{lem_schubertvar_determined_by_projections} and then Lemma \ref{lem_sub_var_under_pl} we get
	\begin{align*}
		C_{I}^{\mathrm{gr}}=\left(\prod\limits_{i\in [n]}\projectgr{i}\right)^{-1}\left(\prod\limits_{i\in [n]}\projectgr{i}\left(C_I^\mathrm{gr}\right)\right)=\left(\prod\limits_{i\in [n]}\projectgr{i}\right)^{-1}\left(\prod\limits_{i\in [n]}\gr{\Lambda_i}_{\res}\cap \project{i}\left(C_I^\mathrm{pr}\right)\right)
	\end{align*} 
	and by looking at the injections in the diagram
	\begin{align*}
	\xymatrix{
	 	\Mgr{k}{\Gamma^{\mathrm{st}}}\ar[r]^{\mathrm{Pl}_{\Gamma^{\mathrm{st}}}} \ar[d]^{\prod\limits_{i\in [n]}\projectgr{i}} 	& \M{\bigwedge^k\Gamma^{st}}\ar[d]^{\prod\limits_{i\in [n]}\project{i}} \\
	 	\prod\limits_{i\in [n]}\gr{\Lambda_i}\ar[r]^{\prod\limits_{i\in [n]}\mathrm{Pl}_{\Lambda_i}} 		& \prod\limits_{i\in [n]}\Projsp(\bigwedge^k\Lambda_i)	
	}
	\end{align*}
	we compute $C^{\mathrm{gr}}_I$ to be 
	\begin{align*}
		\left(\prod\limits_{i\in [n]}\projectgr{i}\right)^{-1}\left(\prod\limits_{i\in [n]}\gr{\Lambda_i}_{\res}\cap \project{i}\left(C_I^\mathrm{pr}\right)\right) 
		&=\left(\prod\limits_{i\in [n]}\project{i}\right)^{-1}\left(\prod\limits_{i\in [n]}\project{i}\left(C_I^\mathrm{pr}\right)\right)\cap \Mgr{k}{\Gamma^{\mathrm{st}}}_{\res}
	\end{align*}
	which clearly contains $C_I^\mathrm{pr} \cap \Mgr{k}{\Gamma^{\mathrm{st}}}_{\res}$.
	We have shown that we get $C^{\mathrm{gr}}_I=C_I^\mathrm{pr} \cap \Mgr{k}{\Gamma^{\mathrm{st}}}_{\res}$ for all $I\in \binom{[n]}{k}$. But since all of the irreducible components of $\M{\bigwedge^k\Gamma^{\mathrm{st}}}_{\res}$ and $\Mgr{k}{\Gamma^{\mathrm{st}}}_{\res}$ are of the form $C_I^\mathrm{pr}$ and $C_I^\mathrm{gr}$ for some $I\in \binom{[n]}{k}$ we have shown the result.
	\end{proof}

\section{A candidate for a semi-stable resolution of $\loc$}
\label{sec_candidate}
With the notation from the last chapter it is an interesting problem to find a semi-stable resolution of the Mustafin variety $\Mgr{k}{\Gamma^{\mathrm{st}}}$ for the standard lattice chain $\Gamma^{\mathrm{st}}$. In general it is not known to be possible, but in Section \ref{sec_mustafin_varieties} we have proven that for $k=1$ the Mustafin variety $\M{\overline{\Gamma}}$ is indeed semi-stable for every convex set $\overline{\Gamma}$ of lattice classes (cf. \cite{FA01}). This generalises the classical case $\M{\Gamma^{\mathrm{st}}}$ of Drinfeld. In \cite{GEN00} a semi-stable resolution for $n\leq 6$ was constructed for a symplectic analogue of the problem via a blow-up of the Grassmannian of isotropic submodules in Schubert subvarieties of the special fiber. Adapting this idea as indicated in \cite[remark (3) following Theorem 2.4.2]{GEN00} we arrive at a candidate for a semi-stable resolution as follows.
\begin{definition}
	We set $\cangen_0\defeq \gr{\Lambda_0}$ and inductively for $1\leq i<(n-k)k$ define $\cangen_i$ to be the blow-up of $\cangen_{i-1}$ in the union of the strict transforms of the Schubert varieties of dimension $i-1$ in the special fiber of $\gr{\Lambda_0}$. The last blow-up $\cangen_{(n-k)k-1}$ will be denoted by $\cangen$.
\end{definition}
 Now two questions arise. 
\begin{enumerate}
	\item Is the blow-up $\cangen$ semi-stable?
	\item Does the map $\cangen\rightarrow \gr{\Lambda_0}$ factor through the Mustafin variety $\Mgr{k}{\Lambda_0}$?
\end{enumerate}
It is well known that the singular locus $X_I^{\mathrm{sing}}$ of a Schubert variety is again a union of Schubert varieties. This gives some hope that the centers in this sequence are in fact smooth. And to analyse the first question Genestier has proven Lemma \ref{lem_gen_semi-stable} showing that blow-ups preserve semi-stability when the centers lie in the special fiber, are smooth over $\res$ and intersect the singular locus nicely.\\
Unfortunately in general the answer to the second question is no. In \cite[Appendix D]{Gor19} we explicitly calculate the case $n=5$ and $k=2$. The result is semi-stable, but does not factor through the Mustafin variety $\Mgr{k}{\Gamma^{\mathrm{st}}}$. To emphasise that this computations can easily be done by hand, let us shortly describe the steps involved. First we recall that by the universal property of the join, a morphism to $\Mgr{k}{\Gamma^{\mathrm{st}}}$ is the same as maps to all factors $\gr{\Lambda_i}$ agreeing on the common generic fiber. Now we simply have to check whether the isomorphism $\cangen_\quot \rightarrow \gr{\Lambda_0}_\quot\rightarrow \gr{\Lambda_i}_\quot$ extends to morphism $\cangen\rightarrow \gr{\Lambda_i}_\quot$ for all $i$ and this can easily be checked by checking the vanishing of certain determinants. \\
 In the following we construct a blow-up $\can$ of $\gr{\Lambda_0}$ as the strict transform of the Pl\"ucker embedding 
$\gr{\Lambda_0}\rightarrow\Projsp(\bigwedge^k\Lambda_0)$ under the blow-up $\projectbar{0}$. We also define the strict transform $\cangenstr$ of the projection $\projectgr{0}$ under the blow-up $\cangen\rightarrow\gr{\Lambda_0}$. 
Then we will show that the blow-up $\can\rightarrow\gr{\Lambda_0}$ factors through $\Mgr{k}{\Gamma^{\mathrm{st}}}$ and $\cangenstr$. In summary we will construct the following commutative diagram:
\begin{align*}
\xymatrix{
 \cangen \ar[ddrr]	&\cangenstr\ar[l]\ar[dr] &\can\ar@{..>}[d]\ar[r]\ar@{..>}[l]\				& \M{\overline{\bigwedge^k\Gamma^{\mathrm{st}}}}\ar^{\projbar}[d]\ar@/^3.0pc/^{\projectbar{0}}[dd]\\
							&	&\Mgr{k}{\Gamma^{\mathrm{st}}}\ar[r]\ar[d]^{\projectgr{0}} 	& \M{\bigwedge^k \Gamma^{\mathrm{st}}}\ar^{\project{0}}[d] \\
								&&\gr{\Lambda_0}\ar[r]						& \Projsp(\bigwedge^k\Lambda_0 )
}	
\end{align*}
In the second part of this chapter we will prove some relations between the candidates. The easiest is to see that if $\cangen\rightarrow\gr{\Lambda_0}$ factors through $\Mgr{k}{\Gamma^{\mathrm{st}}}$ then $\cangenstr\rightarrow\cangen$ is an isomorphism. We will see that $\cangenstr$ restricts to the semi-stable resolution in a neighbourhood of the most singular point for $n=5$ and $k=2$ constructed in \cite{G04}. In particular $\cangenstr$ is a semi-stable resolution whenever $\cangen$ is and  in at least one additional case.\\
If we further assume Conjecture \ref{conj_convexcl_bijection} holds for $n$ and $k$ (e.g. $k=2$ or $n\leq 7$), Theorem \ref{thm_new_candidate_factor_through_gen} will show that under some technical conditions, $\can$ and $\cangen$ coincide whenever $\cangen$ is semi-stable and the map $\cangen\rightarrow\gr{\Lambda_0}$ factors through the local model $\Mgr{k}{\Gamma^{\mathrm{st}}}$.

\subsection{Construction of $\can$, $\cangenstr$ and related objects }

For the rest of this section we will fix a set $\{\Lambda^{\mathrm{pl}}_0,\dots,\Lambda^{\mathrm{pl}}_{\binom{n}{k}-1} \}$ of representatives of the classes in $\overline{\bigwedge^k \Gamma^{\mathrm{st}}}$ such that $\{[\Lambda^{\mathrm{pl}}_0]\dots, [\Lambda^{\mathrm{pl}}_i]\}$ is convex for every $i\in \left[\binom{n}{k} \right]$ and $\Lambda^{\mathrm{pl}}_0= \bigwedge^k\Lambda_0$. Due to Lemma \ref{lem_convexreduction} it is possible to find such representatives.

\begin{definition}
\label{def_candidate_for_resolution}
	Set $\can_0\defeq \gr{\Lambda_0}$ and consider $\can_0$ via the Pl\"ucker embedding as a closed subscheme of $\Projsp(\bigwedge^k\Lambda_0 )=\M{\{[\bigwedge^k\Lambda_0]\}}$. For $1\leq i<\binom{n}{k}$ define $\can_i\subseteq \M{\{[\Lambda^{\mathrm{pl}}_0],\dots, [\Lambda^{\mathrm{pl}}_i]\}}$ to be the strict transform $\str{ \can_{i-1}}$ under the blow-up 
	\begin{align*}
		\M{\{[\Lambda^{\mathrm{pl}}_0],\dots, [\Lambda^{\mathrm{pl}}_i]\}} \rightarrow \M{\{[\Lambda^{\mathrm{pl}}_0],\dots, [\Lambda^{\mathrm{pl}}_{i-1}]\}}.
	\end{align*}
	 We denote $\can_{\binom{n}{k}-1}$ by
	$\can$.
\end{definition}

\begin{lemma}
\label{lem_can_factor_cangen}
	The sequence $\can\rightarrow\gr{\Lambda_0}$ of blow-ups defined above factors through the projection $\Mgr{k}{\Gamma^{\mathrm{st}}}\rightarrow\gr{\Lambda_0}$.
\end{lemma}
\begin{proof}
	Since the centers of the blow-ups are all contained in the special fiber the generic fiber of $\can$ is still dense and maps isomorphically to $\gr{\Lambda_0}_{\quot}$. Hence $\can$ is the closure in the Mustafin variety $\M{\overline{\bigwedge^k\Gamma^{\mathrm{st}}}}$ of the image under the Pl\"ucker embedding 
	\begin{align*}
		\gr{\Lambda_0}_{\quot}\rightarrow \Projsp(\bigwedge^k\Lambda_0 )_{\quot}= \M{\overline{\bigwedge^k\Gamma^{\mathrm{st}}}}_{\quot}.
	\end{align*} 
	 Now since $\projbar(\can_{\quot})$ is the image of $\Mgr{k}{\Gamma^{\mathrm{st}}}_{\quot}$ under the embedding of $\Mgr{k}{\Gamma^{\mathrm{st}}}$ in\linebreak $\M{\bigwedge^k\Gamma^{\mathrm{st}}}$ and $\Mgr{k}{\Gamma^{\mathrm{st}}}$ is closed by construction, the map $\projbar\vert_{\can}$ will factor through $\Mgr{k}{\Gamma^{\mathrm{st}}}$.
\end{proof}

\begin{definition}
\label{def_gen_under_pl}
	Similarly to the construction of the blow-up $\cangen$ we can set $\cangenpl_{0}\defeq\Projsp(\bigwedge^k\Lambda_0)$ and for $1\leq i<(n-k)k$ inductively define $\cangenpl_{i}$ to be the blow-up of $\cangenpl_{i-1}$ in the union of strict transforms of the linear subspaces in $\Projsp(\Lambda_0 )_\res$ corresponding to Schubert varieties in $\gr{\Lambda_0}_\res$ of dimension $i-1$.
\end{definition}

\begin{rem}
\label{rem_gen_under_pluecker}
	Using the relation in Lemma \ref{lem_sub_var_under_pl} between the Schubert varieties in $\gr{\Lambda_0}_{\res}$ and the linear subspaces of $\Projsp(\bigwedge^k \Lambda_0)$, the blow-up $\cangen_j\rightarrow \cangen_{j-1}$ agrees with the strict transform of $\cangen_{j-1}\rightarrow \cangenpl_{j-1}$ under the blow-up $\cangenpl_j\rightarrow\cangenpl_{j-1}$.
\end{rem}

\begin{lemma}
\label{lem_gen_proj_totaltransforms}
	The blow-up $\cangenpl\rightarrow \Projsp(\bigwedge^k\Lambda_0 )$ is the successive blow-up of the total transforms of the linear subspaces $\Projsp(V_I)$ for $I\in \binom{[n]}{k}$ in any order.
\end{lemma}

\begin{proof}
	Fix an integer $i\in [(n-k)k]$. We need to show that the blow-up $\cangenpl_{i+1}\rightarrow\cangenpl_{i}$ defined in Definition \ref{def_gen_under_pl} is the sequence blow-ups in the total transforms of the linear subspaces corresponding to Schubert varieties of dimension $i$. First we note that the intersection of two distinct linear subspaces corresponding to Schubert varieties of dimension $i$ are linear subspaces corresponding to Schubert varieties of dimension $i-1$ and hence the strict transforms of the linear subspaces corresponding to Schubert varieties of dimension $i$ are disjoint in $\cangenpl_{i}$. Therefore the blow-up $\cangenpl_{i+1}\rightarrow\cangenpl_{i}$ can be split up as the chain of blow-ups in the individual strict transforms of the linear subspaces.\\
	Now by induction the strict transforms of linear subspaces corresponding to Schubert varieties are smooth since they are blow-ups of smooth schemes over a field in smooth centers. But the strict transforms in $\cangenpl_{i-1}$ of a linear subspace $L_i$ corresponding to a Schubert variety of dimension $i$ and a linear subspace $L_{i-1}$ corresponding to a Schubert variety of dimension $i-1$ are either disjoint or $L_{i-1}$ is contained in $L_i$. Using Lemma \ref{lem_monoidaltransorm} we identify the blow-up $\bl{\str{L_i}}{\bl{L_{i-1}}{\cangenpl_i}}$ with $\bl{\tot{L_{i-1}}}{\bl{L_{i}}{\cangenpl_i}}$. Since for blow-ups in total transforms the order does not matter, these blow-ups also agree with $\bl{\tot{L_i}}{\bl{L_{i-1}}{\cangenpl_i}}$.
\end{proof}

\begin{corollary}
\label{cor_new_candidate_factor_through_genestier}
	The blow-up $\M{\overline{\bigwedge^k\Gamma^{\mathrm{st}}}}\rightarrow \Projsp( \bigwedge^k\Lambda_0)$ factors through $\cangenpl\rightarrow\Projsp( \bigwedge^k\Lambda_0)$.
\end{corollary}

\begin{proof}
	First recall from Proposition \ref{prop_im_of_irr_comp} that for $I\in\binom{[n]}{k}$ the inverse image of the linear subspace $\Projsp(V_I)$ in $\M{\overline{\bigwedge^k\Gamma^{\mathrm{st}}}}$ is a union of irreducible components. From Proposition \ref{prop_convex_mustafin_semi-stable} we know that $\M{\overline{\bigwedge^k\Gamma^{\mathrm{st}}}}$ is semi-stable and hence unions of irreducible components of the special fiber are effective Cartier divisors. Using the sequence of blow-ups from Lemma \ref{lem_gen_proj_totaltransforms} the universal property of the blow-ups inductively gives a factorisation of the projection $\M{\overline{\bigwedge^k\Gamma^{\mathrm{st}}}}\rightarrow\Projsp(\bigwedge^k\Lambda_0 )$ over $\cangenpl\rightarrow\Projsp(\bigwedge^k\Lambda_0 )$.
\end{proof}

\begin{definition}
\label{def_cangenstr}
	For the sequence of blow-ups $\cangenpl \rightarrow \Projsp(\bigwedge^k \Lambda_0)$ let us denote the sequence of strict transforms of the projection $\M{\bigwedge^k \Gamma^{\mathrm{st}}}\rightarrow\Projsp(\bigwedge^k\Lambda_0)$ by $\cangenstrpl \rightarrow\M{\bigwedge^k\Gamma^{\mathrm{st}}}$. More precisely we set $\cangenstrpl_0=\M{\bigwedge^k \Gamma^{\mathrm{st}}}$ and inductively construct $\cangenstrpl_{i+1}$ as the blow-up $\bl{\mathrm{pr}^{-1}(Z)}{\cangenstrpl_i}$ of $\cangenstrpl_i$ in the inverse image of the center $Z$ of the blow-up $\cangenpl_{i+1}\rightarrow\cangenpl_{i}$ under the projection $\mathrm{pr}\colon\cangenstrpl_i\rightarrow\cangenpl_i$. Similarly we define $\cangenstr\rightarrow \Mgr{k}{\Gamma^{\mathrm{st}}}$ to be the strict transform under the blow-up $\cangen\rightarrow \gr{\Lambda_0}$ of the projection $\Mgr{k}{\Gamma^{\mathrm{st}}}\rightarrow \gr{\Lambda_0}$.
\end{definition}

\begin{lemma}
\label{lem_convex_factor_cangenstrpl}
	The blow-up $\M{\overline{\bigwedge^k\Gamma^{\mathrm{st}}}}\rightarrow \Projsp( \bigwedge^k\Lambda_0)$ factors through $\cangenstrpl\rightarrow\Projsp( \bigwedge^k\Lambda_0)$.
\end{lemma}
\begin{proof}
	Using that the inverse images in $\M{\overline{\bigwedge^k\Gamma^{\mathrm{st}}}}$ of the center of the blow-ups $\cangenstrpl_{i+1}\rightarrow \cangenstrpl_{i}$ and $\cangenpl_{i+1}\rightarrow \cangenpl_{i}$ are the same, we can use the same argument as in the proof of  Corollary \ref{cor_new_candidate_factor_through_genestier}.
\end{proof}

\begin{corollary}
	The blow-up $\can\rightarrow \gr{\Lambda_0}$ factors through $\cangen\rightarrow\gr{\Lambda_0}$.
\end{corollary}

\begin{proof}
	The map $\can\rightarrow \M{\overline{\bigwedge^k\Gamma^{\mathrm{st}}}}$ is by construction the strict transform of the Pl\"ucker embedding $\gr{\Lambda_0}\rightarrow\Projsp(\bigwedge^k\Lambda_0 )$ and hence is a closed immersion. Similarly $\cangenstr\rightarrow\cangenstrpl$ is a closed immersion. Now since $\can$ is flat the restriction of $\M{\overline{\bigwedge^k\Gamma^{\mathrm{st}}}}\rightarrow\cangenstrpl$ to $\can$ factors through $\cangenstr$.
\end{proof}

\subsection{Comparisons of the candidates}
\label{sec_comparisons}
Before we go on and prove some relation between the objects constructed above, let us summarise  all of them with the maps between them in the following diagram: 
\begin{align*}
\xymatrix{
 \cangenstr \ar[d]\ar[dr]	&\can\ar[d]\ar[r]\ar[l] 						& \M{\overline{\bigwedge^k\Gamma^{\mathrm{st}}}}\ar[d]\ar[r] &  \cangenstrpl  \ar[dl] \ar[d]\\
	\cangen\ar[dr]			&\Mgr{k}{\Gamma^{\mathrm{st}}}\ar[r]\ar[d] 	& \M{\bigwedge^k \Gamma^{\mathrm{st}}}\ar[d] & \cangenpl\ar[dl] \\
							&\gr{\Lambda_0}\ar[r]						& \Projsp(\bigwedge^k\Lambda_0 )
}	
\end{align*}
For the case that the map $\cangen\rightarrow \gr{\Lambda_0}$ is factoring over $\Mgr{k}{\Gamma^{\mathrm{st}}}$ by construction we have an identification $\cangen=\cangenstr$. This is the first an easiest relation.

\begin{lemma}
\label{lem_irr_comp_of_can_in_cangenstr}
	Assume Conjecture \ref{conj_convexcl_bijection}. Then the image in $\cangenstrpl_{\res}$ of an irreducible component $C$ of the special fiber $\M{\overline{\bigwedge^k\Gamma^{\mathrm{st}}}}_{\res}$ is an irreducible component.
\end{lemma}

\begin{proof}
	For an irreducible component $C$ of $\M{\overline{\bigwedge^k\Gamma^{\mathrm{st}}}}_{\res}$ the image in $\cangenstrpl_{\res}$ is irreducible and hence lies in an irreducible component. But the images in $\M{{\bigwedge^k\Gamma^{\mathrm{st}}}}_{\res}$ of two different irreducible components of $\M{\overline{\bigwedge^k\Gamma^{\mathrm{st}}}}_{\res}$ are not contained in the same irreducible component. Hence for an irreducible component of $\cangenstrpl_{\res}$ there is at most one irreducible component mapping to it. Now since the map $\M{\overline{\bigwedge^k\Gamma^{\mathrm{st}}}}_{\res} \rightarrow \cangenstrpl_{\res}$ is surjective the image of $C$ is a full irreducible component.
\end{proof}

\begin{prop}
\label{prop_cangenstr_semistable_then_isom}
	 Assume Conjecture \ref{conj_convexcl_bijection} and that $\cangenstrpl$ is semi-stable. Then the projection $\M{\overline{\bigwedge^k\Gamma^{\mathrm{st}}}} \rightarrow \cangenstrpl$ of Lemma \ref{lem_convex_factor_cangenstrpl} is an isomorphism. In this case also the projection $\can\rightarrow \cangenstr$ is an isomorphism.
\end{prop}

\begin{proof}
	Let us assume that for $i\in \left[\binom{n}{k}\right]$ the map $\cangenstrpl\rightarrow \Projsp(\bigwedge^k\Lambda_0 )$ factors through the sequence of blow-ups $\M{\{[\Lambda^{\mathrm{pl}}_0]\dots, [\Lambda^{\mathrm{pl}}_{i}]\}}\rightarrow\Projsp(\bigwedge^k\Lambda_0 )$ defined in the beginning of this chapter. Denote the center of the blow-up $\M{\{[\Lambda^{\mathrm{pl}}_0]\dots, [\Lambda^{\mathrm{pl}}_{i+1}]\}}\rightarrow \M{\{[\Lambda^{\mathrm{pl}}_0]\dots, [\Lambda^{\mathrm{pl}}_{i}]\}}$ by $Z_{\Lambda^{\mathrm{pl}}_{i+1}}$. Then the inverse image of $Z_{\Lambda^{\mathrm{pl}}_{i+1}}$ in $\M{\overline{\bigwedge^k\Gamma^{\mathrm{st}}}}_{\res}$ is by Lemma \ref{lem_inverse_of_center} a union of irreducible components. Hence by Lemma \ref{lem_irr_comp_of_can_in_cangenstr} also the inverse image of $Z_{\Lambda^{\mathrm{pl}}_{i+1}}$ in $\cangenstrpl_{\res}$ is a union of irreducible components. Since $\cangenstrpl$ is semi-stable by assumption, unions of irreducible components of its special fiber are effective Cartier divisors. Therefore the map $\cangenstrpl\rightarrow \Projsp(\bigwedge^k\Lambda_0 )$ factors through $\M{\{[\Lambda^{\mathrm{pl}}_0]\dots, [\Lambda^{\mathrm{pl}}_{i+1}]\}}\rightarrow\Projsp(\bigwedge^k\Lambda_0 )$. By induction we now have an inverse for the map $\M{\overline{\bigwedge^k\Gamma^{\mathrm{st}}}} \rightarrow \cangenstrpl$.\\
	For the second statement we recall that by construction $\can$ and $\cangenstr$ are the strict transforms of the Pl\"ucker embedding under the blow-up $\M{\overline{\bigwedge^k\Gamma^{\mathrm{st}}}}=\cangenstrpl\rightarrow \M{\bigwedge^k\Gamma^{\mathrm{st}}}$ hence they agree. 
\end{proof}

\begin{prop}
\label{prop_semistable_isom}
	Assume Conjecture \ref{conj_convexcl_bijection}. Then $\can\rightarrow \cangenstr$ is an isomorphism whenever $\cangenstr$ is semi-stable and for every irreducible component $C$ of $\cangenstrpl_{\res}$ the intersection $C\cap\cangenstr$ is a union of irreducible components.
\end{prop}

\begin{proof}
	The proof argues in the same way as for the lemma above. For convenience let us recall the arguments.\\
	Let us assume that for $i\in \left[\binom{n}{k}\right]$ the map $\cangenstr\rightarrow \gr{ \Lambda_0 }$ factors through the blow-up\linebreak $\can_i \rightarrow\gr{ \Lambda_0 }$ with $\can_i$ as in Definition \ref{def_candidate_for_resolution}. Denote the center of the blow-up $\can_{i+1}\rightarrow \can_i$ by $Z_{i+1}$.\\

	 Using the notation in the proof of the proposition above we write $Z_{i+1}$ as the intersection $\cangenstr\cap Z_{\Lambda^{\mathrm{pl}}_{i+1}}$. Hence again the inverse image of $Z_{i+1}$ in $\can_{\res}$ is the intersection of a union of irreducible components of $\cangenstrpl_{\res}$ with $\cangenstr$. Using the hypothesis this is an intersection of irreducible components of $\cangenstr_{\res}$. Since $\cangenstr$ is semi-stable by assumption, unions of irreducible components of its special fiber are effective Cartier divisors. Therefore the map $\cangenstr\rightarrow \gr{\Lambda_0 }$ factors through $\can_{i+1}\rightarrow\gr{\Lambda_0 }$. By induction we now have an inverse for the map $\can \rightarrow \cangenstr$.
\end{proof}

\begin{thm}
\label{thm_new_candidate_factor_through_gen}
	Assume Conjecture \ref{conj_convexcl_bijection} and that the blow-up $\cangen\rightarrow \gr{\Lambda_0}$ factors through\linebreak $\Mgr{k}{\Gamma^{\mathrm{st}}}$. Then the map ${\can}\rightarrow \cangen$ of Corollary \ref{cor_new_candidate_factor_through_genestier} is an isomorphism in any of the following cases:
	\begin{enumerate}
		\item $\cangenstrpl$ is semi-stable
		\item $\cangenstr$ is semi-stable and for every irreducible component $C$ of $\cangenstrpl_{\res}$ the intersection $C\cap\cangenstr$ is a union of irreducible components
	\end{enumerate}
\end{thm}

\begin{proof}
	Since by hypothesis $\cangen\rightarrow \gr{\Lambda_0}$ factors through $\Mgr{k}{\Gamma^{\mathrm{st}}}$ the map $\cangenstr\rightarrow\cangen$ is an isomorphism by the universal property of blow-ups. By Proposition \ref{prop_semistable_isom} or Proposition \ref{prop_cangenstr_semistable_then_isom} we now see that ${\can}\rightarrow \cangen$ is an isomorphism.
\end{proof}

As mentioned before in the case of $n=5$, $k=2$ the candidate $\cangen$ of Genestier is indeed semi-stable, but the blow up $\cangen\rightarrow \gr{\Lambda_0}$ does not factor through $\Mgr{k}{\Gamma^{\mathrm{st}}}$. Hence it does not give a semi-stable resolution. But for this case G\"ortz constructed in \cite{G04} a semi-stable resolution $\cangoe$ in a neighbourhood of the worst singularity by blowing up irreducible components of $\Mgr{k}{\Gamma^{\mathrm{st}}}$. We will prove that in this case also the candidate $\cangenstr$ is given by these blow-ups of irreducible components and hence is indeed a semi-stable resolution.

\begin{lemma}
\label{lem_cangenstr_bl_of_irrcom}
	For $n=5$ and $k=2$ the blow-up $\cangenstr\rightarrow \Mgr{k}{\Gamma^{\mathrm{st}}}$ defined above can be alternatively constructed in the following way:
	\begin{enumerate}
		\item $\cangenstr_0$ is set to be $\Mgr{k}{\Gamma^{\mathrm{st}}}$
		\item $\cangenstr_1$ is the blow-up of $\cangenstr_0$ in the irreducible component surjecting to the $0$-dimensional Schubert variety in $\gr{\Lambda_0}_{\res}$
		\item $\cangenstr_i$ is the blow-up of $\cangenstr_{i-1}$ in the union of all strict transforms of irreducible component surjecting to $i-1$-dimensional Schubert varieties in $\gr{\Lambda_0}_{\res}$
	\end{enumerate}
\end{lemma}

\begin{proof}
	Recall that by definition $\cangen_{i+1}$ is the blow-up of $\cangen_{i}$ in union of the strict transform of Schubert varieties of dimension $i$. We will inductively show that $\cangenstr_{i+1}$ is the strict transform of the map $\mathrm{pr}_i\colon \cangenstr_{i}\rightarrow \cangen_{i}$ under the blow-up $\cangen_{i+1} \rightarrow\cangen_{i}$.\\
	We have to show that for the strict transform $\str{X_I}$ of a Schubert variety $X_I$ of dimension $i$ under the blow-up $\cangen_i\rightarrow \gr{\Lambda_0}$ the inverse image $\mathrm{pr}_i^{-1}(\str{X_I})$ in $\cangenstr_i$ is the strict transform $\str{C_I}$ of the irreducible component $C_I$ of $\Mgr{k}{\Gamma^{\mathrm{st}}}_{\res}$ under the blow-up $\cangenstr_i\rightarrow\Mgr{k}{\Gamma^{\mathrm{st}}}$. But $\cangen\rightarrow\bl{\str{X_I}}{ \cangen_{i}}$ is surjective and the exceptional divisor $E$ in $\bl{\str{X_I}}{ \cangen_{i}}$ is an irreducible component in the special fiber. Hence there exists an irreducible component in $\cangen_{\res}$ surjecting to $E$. Again since $\cangenstr\rightarrow \cangen$ is surjective there exists an irreducible component of $\cangenstr_{\res}$ surjecting to $E$. Now the image of $E$ in $\bl{\mathrm{pr}_i^{-1}(\str{X_I})}{ \cangenstr_{i}}$ is the exceptional divisor and its image in $\cangen_i$ is an irreducible component.\\
	 Using that the irreducible components in $\cangenstr_{i,\res}$ are precisely the strict transforms of irreducible components of $\Mgr{k}{\Gamma^{\mathrm{st}}}$ and that the irreducible component surjecting to $X_I$ is unique, we have shown the claim.
\end{proof}

\begin{thm}
\label{thm_5_2_goe}
	For $n=5$ and $k=2$ the candidate $\cangenstr$ restricts to the semi-stable resolution defined in \cite{G04} in a neighbourhood of the worst singularity of $\loc$.
\end{thm}

\begin{proof}
	In Lemma \ref{lem_cangenstr_bl_of_irrcom} we have described $\cangenstr$ by a sequence of blow-ups of $\Mgr{2}{\Gamma^{\mathrm{st}}}$ in irreducible components of $\Mgr{2}{\Gamma^{\mathrm{st}}}_{\res}$. Since the resolution $\cangoe$ is constructed in the same way, but blowing up just the strict transforms of irreducible components not corresponding to a lattice in $\Gamma^{\mathrm{st}}$ we have locally a factorisation of $\cangenstr\rightarrow \Mgr{2}{\Gamma^{\mathrm{st}}}$ over $\cangoe$. Using that $\cangoe$ is semi-stable we now see that we get an inclusion $\cangoe \rightarrow \cangenstr $. 
\end{proof}

\bibliographystyle{alpha}
\bibliography{literatur}

\begin{thebibliography}{CHSW11}

\bibitem[AL17]{AL17}
Marvin {Anas Hahn} and Binglin {Li}.
\newblock {Mustafin varieties, moduli spaces and tropical geometry}.
\newblock {\em arXiv e-prints}, page arXiv:1707.01216, Jul 2017.

\bibitem[CHSW11]{CHSW11}
Dustin Cartwright, Mathias H\"{a}bich, Bernd Sturmfels, and Annette Werner.
\newblock Mustafin varieties.
\newblock {\em Selecta Math. (N.S.)}, 17(4):757--793, 2011.

\bibitem[Deo77]{D77}
Vinay~V. Deodhar.
\newblock Some characterizations of {B}ruhat ordering on a {C}oxeter group and
  determination of the relative {M}\"{o}bius function.
\newblock {\em Invent. Math.}, 39(2):187--198, 1977.

\bibitem[dJ96]{dJ96}
A.~J. de~Jong.
\newblock Smoothness, semi-stability and alterations.
\newblock {\em Inst. Hautes \'{E}tudes Sci. Publ. Math.}, (83):51--93, 1996.

\bibitem[Fal01]{FA01}
Gerd Faltings.
\newblock Toroidal resolutions for some matrix singularities.
\newblock In {\em Moduli of abelian varieties ({T}exel {I}sland, 1999)}, volume
  195 of {\em Progr. Math.}, pages 157--184. Birkh\"{a}user, Basel, 2001.

\bibitem[Gen00]{GEN00}
Alain Genestier.
\newblock Un mod{\`e}le semi-stable de la vari\'{e}t\'{e} de {S}iegel de genre
  3 avec structures de niveau de type {$\Gamma_0(p)$}.
\newblock {\em Compositio Math.}, 123(3):303--328, 2000.

\bibitem[G{\"o}r01]{G01}
Ulrich G{\"o}rtz.
\newblock On the flatness of models of certain {S}himura varieties of
  {PEL}-type.
\newblock {\em Math. Ann.}, 321(3):689--727, 2001.

\bibitem[G{\"{o}}r03]{Goe03}
Ulrich G{\"{o}}rtz.
\newblock On the flatness of local models for the symplectic group.
\newblock {\em Adv. Math.}, 176(1):89--115, 2003.

\bibitem[G{\"o}r04]{G04}
Ulrich G{\"o}rtz.
\newblock Computing the alternating trace of {F}robenius on the sheaves of
  nearby cycles on local models for {$\rm GL_4$} and {$\rm GL_5$}.
\newblock {\em J. Algebra}, 278(1):148--172, 2004.

\bibitem[Gor19]{Gor19}
Felix Gora.
\newblock {\em Local models, Mustafin varieties and semi-stable resolutions}.
\newblock PhD thesis, 2019.

\bibitem[H{\"a}b11]{Ha11}
M.~H{\"a}bich.
\newblock {\em Degenerations of flag varieties}.
\newblock PhD thesis, 2011.

\bibitem[H{\"a}b14]{Hae14}
Mathias H{\"a}bich.
\newblock Mustafin degenerations.
\newblock {\em Beitr. Algebra Geom.}, 55(1):243--252, 2014.

\bibitem[Hai05]{Ha05}
Thomas~J. Haines.
\newblock Introduction to {S}himura varieties with bad reduction of parahoric
  type.
\newblock In {\em Harmonic analysis, the trace formula, and {S}himura
  varieties}, volume~4 of {\em Clay Math. Proc.}, pages 583--642. Amer. Math.
  Soc., Providence, RI, 2005.

\bibitem[Har01]{Har01}
Urs~T. Hartl.
\newblock Semi-stability and base change.
\newblock {\em Arch. Math. (Basel)}, 77(3):215--221, 2001.

\bibitem[HP94]{HP94}
W.~V.~D. Hodge and D.~Pedoe.
\newblock {\em Methods of algebraic geometry. {V}ol. {II}}.
\newblock Cambridge Mathematical Library. Cambridge University Press,
  Cambridge, 1994.
\newblock Book III: General theory of algebraic varieties in projective space,
  Book IV: Quadrics and Grassmann varieties, Reprint of the 1952 original.

\bibitem[JSY07]{SJ07}
Michael Joswig, Bernd Sturmfels, and Josephine Yu.
\newblock Affine buildings and tropical convexity.
\newblock {\em Albanian J. Math.}, 1(4):187--211, 2007.

\bibitem[KR00]{KR00}
R.~Kottwitz and M.~Rapoport.
\newblock Minuscule alcoves for {${\rm GL}_n$} and {$G{\rm Sp}_{2n}$}.
\newblock {\em Manuscripta Math.}, 102(4):403--428, 2000.

\bibitem[Kr{\"{a}}03]{Kr03}
N.~Kr{\"{a}}mer.
\newblock Local models for ramified unitary groups.
\newblock {\em Abh. Math. Sem. Univ. Hamburg}, 73:67--80, 2003.

\bibitem[Li18]{Li18}
Binglin Li.
\newblock Images of rational maps of projective spaces.
\newblock {\em Int. Math. Res. Not. IMRN}, (13):4190--4228, 2018.

\bibitem[MS15]{MS15}
Diane Maclagan and Bernd Sturmfels.
\newblock {\em Introduction to tropical geometry}, volume 161 of {\em Graduate
  Studies in Mathematics}.
\newblock American Mathematical Society, Providence, RI, 2015.

\bibitem[Mum72]{Mu72}
David Mumford.
\newblock An analytic construction of degenerating curves over complete local
  rings.
\newblock {\em Compositio Math.}, 24:129--174, 1972.

\bibitem[Mus78]{Mu78}
G.~A. Mustafin.
\newblock Non-{A}rchimedean uniformization.
\newblock {\em Mat. Sb. (N.S.)}, 105(147)(2):207--237, 287, 1978.

\bibitem[Pap00]{Pa00}
Georgios Pappas.
\newblock On the arithmetic moduli schemes of {PEL} {S}himura varieties.
\newblock {\em J. Algebraic Geom.}, 9(3):577--605, 2000.

\bibitem[PR05]{PR05}
G.~Pappas and M.~Rapoport.
\newblock Local models in the ramified case. {II}. {S}plitting models.
\newblock {\em Duke Math. J.}, 127(2):193--250, 2005.

\bibitem[PRS13]{PRS13}
Georgios Pappas, Michael Rapoport, and Brian Smithling.
\newblock Local models of {S}himura varieties, {I}. {G}eometry and
  combinatorics.
\newblock In {\em Handbook of moduli. {V}ol. {III}}, volume~26 of {\em Adv.
  Lect. Math. (ALM)}, pages 135--217. Int. Press, Somerville, MA, 2013.

\bibitem[Ric13]{R13}
Timo Richarz.
\newblock Schubert varieties in twisted affine flag varieties and local models.
\newblock {\em J. Algebra}, 375:121--147, 2013.

\bibitem[RZ96]{RZ96}
M.~Rapoport and Th. Zink.
\newblock {\em Period spaces for {$p$}-divisible groups}, volume 141 of {\em
  Annals of Mathematics Studies}.
\newblock Princeton University Press, Princeton, NJ, 1996.

\end{thebibliography}

\end{document}